\documentclass[a4paper,12pt]{article}
\usepackage{amsmath}
\usepackage{mathtools}
\usepackage{amssymb}
\usepackage{amsthm}
\usepackage{amscd}
\usepackage{booktabs}
\usepackage{verbatim}
\usepackage{eucal}
\usepackage[T1]{fontenc}
\usepackage[utf8]{inputenc}
\newcommand{\Hl}{{\mathbb H}}
\newcommand{\C}{{\mathbb C}}
\newcommand{\R}{{\mathbb R}}
\newcommand{\Q}{{\mathbb Q}}
\newcommand{\Z}{{\mathbb Z}}

\newcommand{\g}{{\frak g}}
\newcommand{\p}{{\frak p}}

\newcommand{\lk}{{\mathfrak k}}

\newcommand{\m}{{\mathfrak m}}

\newcommand{\bA}{{\mathbb A}}
\newcommand{\bV}{{\mathbb V}}

\newcommand{\Hom}{\operatorname{Hom}}
\newcommand{\Sym}{\operatorname{Sym}}

\newcommand{\tr}{\operatorname{tr}}

\newcommand{\cA}{{\cal A}}
\newcommand{\cD}{{\cal D}}
\newcommand{\cG}{{\cal G}}
\newcommand{\cH}{{\cal H}}

\newcommand{\cN}{{\cal N}}
\newcommand{\cP}{{\cal P}}
\newcommand{\cR}{{\cal R}}

\newcommand{\cW}{{\cal W}}

\numberwithin{equation}{section}
\theoremstyle{plain}
 \newtheorem{thm}{Theorem}[section]
 \newtheorem{prop}[thm]{Proposition}
 \newtheorem{lem}[thm]{Lemma}
 \newtheorem{cor}[thm]{Corollary}

\theoremstyle{definition}
 \newtheorem{defn}[thm]{Definition}

 \newtheorem*{asm}{Assumption}
 
 \newtheorem{pf}{Proof}
 
\begin{document}
\begin{title}
{\Large\bf Fourier-Jacobi expansion of automorphic forms generating quaternionic discrete series}
\end{title}
\author{Hiro-aki Narita}
\maketitle
\begin{abstract}
We provide a theory of the Fourier-Jacobi expansion for automorphic forms on simple adjoint groups of some general class. This theory respects the Heisenberg parabolic subgroups, whose unipotent radicals are the Heisenberg groups uniformly explained in terms of the notion of cubic norm structures. Based on this theory of the Fourier expansion, we prove that automorphic forms generating quaternionic discrete series representations automatically satisfy the moderate growth condition except for the cases of the group of $G_2$-type and special orthogonal groups of signature $(4,N)$. This should be called ``K\"ocher principle'' verified already for the case of the quaternion unitary group $Sp(1,q)$ for $q>1$ by the author. We also prove that every term of the Fourier expansion with a non-trivial central character for cusp forms generating quaternionic discrete series has no contribution by the discrete spectrum of the Jacobi group, which is a non-reductive subgroup of the Heisenberg parabolic subgroup. This is obtained by showing that generalized Whittaker functions of moderate growth for the Schr\"odinger representations are zero under some assumption of the separation of variables, which suffices for our purpose to establish such consequence.  

\end{abstract}

\section{Introduction}
As well known classical works of Jacobi, Hecke and Siegel et al. indicate, the arithmetic of quadratic forms naturally leads to detailed studies of elliptic modular forms and holomorphic Siegel modular forms. 
The arithmetic of binary cubic forms is significantly related to automorphic forms on the exceptional group $G_2$ as was pointed out by Gan-Gross-Savin \cite{G-G-S}. However, one should note that there are no holomorphic automorphic forms on $G_2$, which corresponds to the fact that there are no holomorphic discrete series representations of $G_2$. Instead there is a good class of discrete series representations of $G_2$ called quaternionic discrete series representations introduced by Gross-Wallach \cite{G-W}. 
In the crucial work by Gan-Gross-Savin \cite{G-G-S} the multiplicity free theorem for generalized Whittaker model of quaternionic discrete series by Wallach is of fundamental importance, which enables \cite{G-G-S} to discuss the arithmetic of Fourier coefficients of automorphic forms on $G_2$ generating such discrete series. 

The next story begins with the essential work by Pollack \cite{P1}, which succeeded in making the generalized Whittaker models explicit by their functional realization, namely generalized Whittaker functions for quaternionic discrete series. In fact, Pollack \cite{P1} deals with a more general case. Let $\cG$ be a simple adjoint group over $\Q$, whose group $\cG(\R)$ of real points admit quaternionic discrete series representations. Let $F$ be an automorphic form on the adele group $\cG(\bA)$ generating a quaternionic discrete series at the archimedean place. In what follows, as Pollack \cite{P1} did, we exclude the cases where $\cG(\R)$ is isogenous to the special unitary group $SU(2,m)$ of signature $(2,m)$ or to the indefinite symplectic group $Sp(1,q)$ of signature $(1,q)$, which also have quaternionic discrete series. 

Keeping the setting on $\cG$ in \cite{P1}, we can describe the Fourier-Jacobi expansion of an automorphic form $F$ on $\cG(\bA)$ in full generality and without assuming the moderate growth condition. Let $\cN$ be the unipotent radical of the Heisenberg parabolic subgroup of $\cG$. This is a $\Q$-algebraic group equipped with a structure of a Heisenberg group. The group $\cN$ modulo center is isomorphic to the vector part $\cW_J$ of the Heisenberg group associated with a cubic norm structure $J$ introduced in Section \ref{Cubicforms-Groups} (for the notation $\cN$ and $\cW_J$ see Section \ref{FJ-exp}). 
The most rough formulation of the Fourier-Jacobi expansion of $F$ is
\[
F(g)=\sum_{\xi\in\Q}F_{\xi}(g)\quad(g\in\cG(\bA)),
\]
where $F_{\xi}$ denotes the Fourier transformation of $F$ with respect to the central character of $\cN(\bA)$ indexed by $\xi$. The term $F_0$ decomposes into a sum of the Fourier transformations of $F$ by characters of $\cN(\Q)\backslash\cN(\bA)$ parametrized by the $\Q$-rational points $\cW_J(\Q)$ of $\cW_J$. This is due to the triviality of $F_0$ with respect to the center of $\cN$. In this paper we can also describe $F_{\xi}$ for non-zero $\xi$ in terms of such Fourier transformations. More precisely we need a reflection $w_{\alpha}\in\cG(\Q)$ associated with a root $\alpha$~(cf.~Section \ref{FJ-exp}) for such description. 
\begin{thm}[Theorem \ref{adelic-FJ-exp}]\label{adelic-FJ-exp_intro}
For $\xi\in\Q\setminus\{0\}$ let $F_{\xi}$ be as above for a general automorphic form $F$ on $\cG(\bA)$. 
For a character $\chi$ of $\cN(\Q)\backslash\cN(\bA)$, $F_{\chi}$ denotes the Fourier transformation of $F$ by $\chi$. 
Using the coordinate $n(w,t)$ of $\cN(\Q)$ with $w\in \cW_J(\Q)$ and $t\in\Q$ we  have
\[
F_{\xi}(g)=\sum_{\mu\in\Q}\sum_{\chi'_{\xi}}F_{\chi_{\xi}}(w_{\alpha}n((\mu,0,0,0),0)g)=\sum_{\chi''_{\xi}}\Theta_{\chi''_{\xi}}(g)
\]
with
\[
\Theta_{\chi''_{\xi}}(g):=\sum_{\gamma\in\cN_m(\Q)\backslash\cN(\Q)}F_{\chi''_{\xi}}(w_{\alpha}\gamma g),
\]
where $\chi'_{\xi}$~(respectively~$\chi''_{\xi}$) runs over characters of $\cN(\Q)\backslash\cN(\bA)$ parametrized by $(\xi,\beta,\gamma,\delta)\in \cW_J(\Q)$~(respectively~$(\xi,0,\gamma',\delta')\in \cW_J(\Q)$). The function $\Theta_{\chi''_{\xi}}$ above is left invariant by the $\Q$-rational points $\cR_J(\Q)$ of the Jacobi group $\cR_J$~(cf.~Section \ref{FJ-exp}).
\end{thm}
Regarding this theorem we remark that a recent paper \cite{KY} by Kim and Yamauchi includes a theory of the Fourier expansion for the group of $G_2$-type obtained by a different idea. 

If we assume that the multiplicity free property of the automorphic form $F$ with respect to non-trivial characters of the Heisenberg group, we can discuss Fourier coefficients for $F_{\xi}$ with $\xi\not=0$ as well as $F_0$. 
What we should stress under this assumption is that the Fourier coefficients in $F_{\xi}$ are essentially the same as those of $F_0$, which are quite different from what we have experienced in the studies of holomorphic modular forms such as Siegel modular forms etc. As an immediate consequence of this theorem we have the following:
\begin{cor}[Corollary \ref{Determination-by-zeroth-coeff}]
A general automorphic form on $F$ is determined by $F_0$. That is, if two automorphic forms $F$ and $G$ on $\cG(\bA)$ satisfy $F_0=G_0$, we then have $F\equiv G$.
\end{cor}
\noindent
We should remark that this is a generalization of \cite[Proposition 8.4]{G-G-S}. 

Now let us concentrate on automorphic forms generating quaternionic discrete series representations at the archimedean place. 
Pollack \cite{P1} provided the archimedean generalized Whittaker functions explicitly for characters of the Heisenberg group in terms of the K-Bessel function and the quantity arising from cubic norm structures~(cf.~Sections \ref{Cubicforms-Groups}). With these functions we can understand the archimedean part of $F_0$. The Fourier coefficients in the term $F_0$ are defined as the coefficients of the explicitly given Whittaker functions by Pollack \cite{P1} in $F_0$. Due to the fact that cusp forms of level one are determined by $F_0$ for the exceptional group of $G_2$-type~(cf.~\cite[Proposition 8.4]{G-G-S}) the studies after \cite{P1} have concentrated on $F_0$. 

However, by virtue of the general theory of the Fourier-Jacobi expansion above together with Pollack's result, we are able to study the Fourier-Jacobi expansion of automorphic forms generating quaternionic discrete series in detail without assuming the moderate growth condition. In fact, we have obtained the following result~(cf.~Theorem \ref{KoecherPl} part 2):
\begin{thm}[K\"ocher principle]
Let $\cG$ be a simple adjoint group as above, whose real group admit quaternionic discrete series. We exclude the cases of special orthogonal groups $SO(4,N)$~($PSO(4,N)$ for even $N$) and the exceptional group of type $G_2$.  
Automorphic forms on $\cG(\bA)$ generating quaternionic discrete series representations at the archimedean place are automatically of moderate growth. 
\end{thm}
We now remark that Ko\"cher principle has been already verified for $G=Sp(1,q)$ with $q>1$~(cf.~\cite{Na-1}). As a relevant work we also refer to the ``anti-K\"odcher principle'' for automorphic forms on $G=SU(2,2)$ generating middle discrete series representations~(which include quaternionic discrete series) pointed out by Gon~(cf.~\cite{Gy}). 
A rough explanation of the proof for the theorem above is to reduce the proof of the moderate growth property to that of the constant term of the Fourier-Jacobi expansion, that is, the term given by the Fourier transform with respect to the trivial character of the Heisenberg group. This together with Proposition \ref{lower-bound-FJ} explains the reason why we are not able to prove the K\"ocher principle for the special orthogonal groups mentioned above and the exceptional group $G_2$ since the constant term for the cases of  $G_2$~(respectively~the special orthogonal groups) is written in terms of elliptic modular forms~(respectively~holomorphic modular forms on a direct product of the complex upper half plane with a symmetric domain of type $IV$)~(cf.~\cite[Theorem 1.2.1~(2)]{P1}), for which nobody is able to prove the K\"ocher principle as is well known. Here we do not impose the holomorphy around cusps on the defining condition of elliptic modular forms~(or holomorphic modular forms on the complex upper half plane).

In order to understand the archimedean aspect of the Fourier-Jacobi expansion completely it is indispensable to study generalized Whittaker functions for the infinite dimensional unitary representations called Schr\"odinger representations of the Heisenberg group, which exhaust irreducible unitary representations of the Heisenberg group other than characters. As a result of detailed studies on these Whittaker functions we verify that there are no such Whittaker functions contributing to the Fourier-Jacobi expansion of automorphic forms generating quaternionic discrete series. To be precise we study them under some assumption on the separation of variables, which suffices for our purpose~(cf.~Theorem \ref{GenWhittakerFct-gen}). 

This study of generalized Whittaker functions leads us to a detailed understanding of  the Fourier-Jacobi type spherical functions for the real groups $G:=\cG(\R)$ as above like Hirano \cite{Hi-1},~\cite{Hi-2} and \cite{Hi-3}, which can be reformulated as the Fourier-Jacobi models. 
\begin{thm}[Theorem \ref{FJ-model}]\label{FJ-spherical-intro}
Let $G:=\cG(\R)$ be just as above but, for this assertion, we include the cases of the exceptional group $G_2$ and the special orthogonal groups of signature $(4,N)$. 
For any irreducible unitary representations $\rho$ of the Jacobi group $R_J$ with non-trivial central characters~(cf.~Section \ref{Rep-JacobiGp}), 
Fourier-Jacobi type spherical functions of moderate growth for quaternionic discrete series representations are zero. This result is reformulated as
\[
\dim\Hom_{(\g,K)}(\pi_n,C^{\infty}_{\rm mod}\text{-}{\rm Ind}_{R_J}^G\rho)=0
\]
for quaternionic discrete series representation $\pi_n$~(for the notation $\pi_n$ see Section \ref{GWF-definition}), where $C^{\infty}_{\rm mod}\text{-}{\rm Ind}_{R_J}^G\rho$ denotes the subrepresentation of $C^{\infty}$-induction $C^{\infty}\text{-}{\rm Ind}_{R_J}^G\rho$ with the representation space given by the moderate growth sections of $C^{\infty}\text{-}{\rm Ind}_{R_J}^G\rho$. Here we consider the intertwining spaces as $(\g,K)$-modules with the Lie algebra $\g$ of $G$ and the maximal compact subgroup $K$ of $G$.
\end{thm}
We have an explicit formula for generalized Whittaker functions for $\pi_n$ attached to the Schr\"odinger representations under the assumption on the separation of variable mentioned above~(~(cf.~Theorem \ref{GenWhittakerFct-gen}) (1)). They are not of moderate growth. This yields the theorem above. We can say that this leads to another unique feature of automorphic forms generating quaternionic discrete series, which we are going to remark.

Let $F$ be a cusp form on $\cG(\bA)$ generating quaternionic discrete series representation at the archimedean place. In Theorem \ref{adelicFJ-exp-QDS} Fourier-Jacobi expansion of $F$ is given in a further detail. Let $\cR_J$ be the Jacobi group over $\Q$, whose real group we have used in the statement of Theorem \ref{adelic-FJ-exp_intro}~(or Theorem \ref{adelic-FJ-exp}). For $\xi\in\Q\setminus\{0\}$, $F_{\xi}$ is regarded as a square-integrable automorphic form on $\cR_J(\Q)\backslash\cR_J(\bA)$. The $L^2$-space $L^2_{\xi}(\cR_J(\Q)\backslash\cR_J(\bA))$ with the central character parametrized by $\xi$ is a direct sum of the continuous spectrum and the discrete spectrum, the latter of which is a discrete sum of irreducible unitary representations of $\cR_J(\bA)$. 
From Theorem \ref{FJ-spherical-intro} we deduce that $F_{\xi}$ has no contribution from the discrete spectrum. Here thus comes another unique feature of automorphic forms generating quaterninonic discrete series representations:
\begin{thm}[Theorem \ref{adelicFJ-exp-QDS}]\label{FJ-intro-cusp}
Let $F$ be a cusp form as above. For $\xi\not=0$ (the coefficient functions of)  $F_{\xi}$ belongs to the continuous spectrum of the $L^2$-space $L^2_{\xi}(\cR_J(\Q)\backslash \cR_J(\bA))$.
\end{thm}
This result is totally beyond author's expectation. 
The Fourier-Jacobi exansions of holomorphic Siegel cusp forms of degree two are proved to need only the discrete spectrum~(cf.~\cite[Section 4]{B-S}). This is naturally expected for the case of general degree by the well known fact that the Fourier-Jacobi coefficients of non-zero indices are Jacobi cusp forms for the Fourier-Jacobi expansion of holomorphic cusp forms of general degree . In addition, we remark that, according to \cite{Na-2}, we need the discrete spectrum of the Jacobi group as well as the continuous spectrum to describe the Fourier-Jacobi expansions of generic cusp forms on $Sp(2,\R)$, for which we cannot deny the contribution from the discrete spectrum since the result like Theorem \ref{FJ-spherical-intro} would be deniable in view of Hirano \cite{Hi-1} \cite{Hi-2} and \cite{Hi-3}. 

Let us explain the outline of this paper. We begin with analysis on the generalized Whittaker functions, which yields the result on the Fourier-Jacobi type spherical functions or the Fourier-Jacobi models. In Section 2 we study the generalized Whittaker functions attached to Schr\"odinger representations and prove the multiplicity formulas for generalized Whitaker models and Fourier-Jacobi models. 
This section also takes up generalized Whittaker functions for the characters studied by Pollack. Motivated by the proof of the K\"ocher priciple, we study them relaxing the moderate growth condition. The section 3 is then devoted to the proof of the theorem on the Fourier-Jacobi expansion of general automorphic forms and forms generating quaternionic discrete series. The proof of the K\"ocher priciple~(Theorem \ref{KoecherPl}) is taken up and we prove Theorem \ref{FJ-intro-cusp}~(Theorem \ref{adelicFJ-exp-QDS}) in this section. In Section 4 we verify that the interesting examples of cusp forms by Pollack \cite{P3} generate quaternionic discrete series, which can be regarded as an application of the results in Section 2. 
In the appendix we study the Fourier-Jacobi expansion of cusp forms generating quaternionic discrete series in the non-adelic formulation.
\subsection*{Acknowledgement}
This paper would be impossible without detailed discussions with Wee Teck Gan, Aaron Pollack and Takuya Yamauchi. The author's profound gratitude is due to them. 
The author was supported by Grand-in-Aid for Scientific Research (C) 19K03431, Japan Society for the Promotion of Science and by Waseda University Grant for Special Research Projects (Project number: 2022C-095). This work was supported by the Research Institute for Mathematical Sciences, an International Joint Usage/Research Center located in Kyoto University.
\section{Generalized Whittaker functions and Fourier-Jacobi models for quaternionic discrete series representations}. 
This section can be said to be the essential part of the archimedean aspect of this  paper. Our results on the generalized Whittaker functions and Fourier-Jacobi type spherical functions for quaternionic discrete series are indispensable to understand the archimedean aspect of the Fourier-Jacobi expansion completely. In Sections \ref{Cubicforms-Groups} and \ref{Rep-JacobiGp} we take up the fundamental facts on groups and representations to provide the basis for the coming discussion. After examining the generalized Whittaker functions in Sections \ref{GWF-definition}, \ref{GWF-STrep} and \ref{Pf-Prop2.4}, our results are stated in Section \ref{Results-GWF-FJS}. We work over real groups unless otherwise stated.
\subsection{Cubic norm structures and associated groups}\label{Cubicforms-Groups}
Let $G$ be the connected component of the identity for an adjoint real simple Lie group as in Pollack \cite{P1}, namely $G$ is one of simple real Lie groups with the quaternionic structure, other than groups isogenous to the special unitary group $SU(2,m)$ of signature $(2,m)$ or to the quaternion unitary group~(indefinite symplectic group) $Sp(1,q)$ of signature $(1,q)$.  These simple adjoint groups are described by using vector spaces $J$ over $\R$ with the cubic norm structures explained soon and listed as follows~(cf.~\cite[Section 1.2]{P1}):
\begin{enumerate}
\item $J=\R$ for $G$ of type $G_2$,
\item $J=H_3(\R)$ for $G$ of type $F_4$,
\item $J=H_3(\C)$ for $G$ of type $E_6$,
\item $J=H_3(\Hl)$ for $G$ of type $E_7$,
\item $J=H_3(\Theta)$ with the non-split octanion $\Theta$ for $G$ of type $E_8$,
\item $J=\R\times V$ with a quadratic space $V$ of signature $(1,m-3)$, for  $G=SO(m,4)$~(respectively $PSO(m,4)$) if $m$ is odd~(respectively~even).
\end{enumerate}
Here $H_3(C)$ denotes the set of Hermitian matrices with coefficients in the composition algebras $C$. For these we also note that $J$s listed above are known as examples of formally real Jordan algebras and that, associated with $J$ and the  symmetric cone $\Omega\subset J$, there are tube domains
\[
D_J:=\{Z=X+\sqrt{-1}Y\mid X\in J,~Y\in\Omega\},
\]
which are the Riemannian symmetric spaces for the Levi parts $H_J$~(cf.~\cite[Sections 2.2,~2.3]{P1}) of the Heisenberg parabolic subgroups of $G$~(explained later)

Let us review the notion of a cubic norm structure, following \cite[Section 2.1]{P1}. 
We work over a field $F$ of characteristic $0$ for this review. We say that a finite dimensional vector space $J$ over $F$ is equipped with a cubic norm structure if $J$ has a cubic polynomial $N_J: J\rightarrow F$, a quadratic polynomial map $\sharp:J\rightarrow J$, an element $1_J\in J$, and a symmetric bilinear paring $(*,*):J\otimes J\rightarrow F$ called the trace paring, that satisfy the following properties:\\
For $x,y\in J$ we put $x\times y:=(x+y)^{\sharp}-x^{\sharp}-y^{\sharp}$ and denote by $(*,*,*):J\otimes J\otimes J\rightarrow F$ the unique symmetric trilinear form satisfying $(x,x,x)=6N_J(x)$ for all $x\in J$.  Then
\begin{enumerate}
\item $N_J(1_J)=1,~1_J^{\sharp}=1_J$ and $1_J\times x=(1_J,x)\cdot1_J-x$ for all $x\in J$;
\item $(x^{\sharp})^{\sharp}=N(x)x$ for all $x\in J$;
\item $(x,y)=\displaystyle\frac{1}{4}(1_J,1_J,x)(1_j,1_j,y)-(1_J,x,y)$ for all $x,~y\in J$;
\item $N_J(x+y)=N_J(x)+(x^{\sharp},y)+(x,y^{\sharp})+N_J(y)$ for all $x,y\in J$.
\end{enumerate}
As a fundamental reference on this we cite \cite{Mc}. 

The trace paring canonically leads to the dual $J^{\vee}$ of $J$, which can be identified with $J$ itself. We introduce the symplectic vector space
\[
W_J:=F\oplus J\oplus J^{\vee}\oplus F
\]
equipped with the symplectic form
\[
\langle(a,b,c,d),(a',b',c',d')\rangle:=ad'-(b,c')+(c,b')-da'
\]
for $(a,b,c,d),~(a',b',c',d')\in W_J$.

\subsection{Jacobi groups and their irreducible unitary representations}\label{Rep-JacobiGp}
Let $P=H_JN$ be the Heisenberg parabolic subgroup with the Levi part $H_J$ and the unipotent radical $N$. The Levi part includes $H_J^{0,1}$, the connected component of the identify for the semisimple part of $H_J$~(cf.~\cite[Section 7.5]{P1}). We introduce the subgroup
\[
R_J:=H_J^{0,1}\ltimes N
\]
of $P$, which we want to call the Jacobi group. To discuss irreducible unitary representations of $R_J$ we introduce the metaplectic group $Mp(W_J)$ of the real symplectic group $Sp(W_J)$ defined by the symplectic vector space $W_J$, for which note that $Mp(W_J)$ is the unique non-split two fold cover of $Sp(W_J)$. 

The representation theory of Jacobi groups have been developed in a more general context~(for instance see \cite{Su}). We restrict ourselves to unitary representations. Our discussion here is aimed at Theorems \ref{FJ-model} and \ref{adelicFJ-exp-QDS}, for which our limitation is enough. 

The group $N$ is a two step nilpotent Lie group and can be called a Heisenberg group. This Heisenberg group $N$ can be written by the coordinate
\[
n((a,b,c,d),t)\quad((a,b,c,d)\in W_J,~t\in\R),
\]
where $t$ denotes the coordinate of the center. Its composition law is given by
\[
n(w;t)\cdot n(w';t'):=n(w+w',\frac{1}{2}\langle w,w'\rangle+t+t')
\]
for $w,w'\in W_J$ and $t,t'\in\R$. 
The abelian quotient $N/[N,N]$ is viewed as the vector part of the Heisenberg group and can be identified with $W_J$. As for $H_J$ we remark that its connected component $H_J^{0}$ of the identity acts transitively on the tube domain $D_J$ as is shown in the proof of \cite[Proposition 2.3.1]{P1}. 

We put
\[
N_m:=\{n((0,0,c,d),t)\in N\mid c\in J^{\vee},~d,~t\in\R\}.
\]
This is a maximal abelian subgroup of $N$, whose Lie algebra is a polarization subalgebra ~(cf.~\cite[p28]{Co-G-2})~with respect to a linear form of ${\rm Lie}(N)$~(the Lie algebra of $N$) non-trivial only on the center of ${\rm Lie}(N)$. 
We introduce the character $\chi_{\xi}$ of $N_m$ defined by
\[
\psi_{\xi}(n((0,0,c,d),t))=\exp(2\pi\sqrt{-1}\xi t).
\]
For $\xi\in\R\setminus\{0\}$ we define a Schr\"odinger representation  of $C^{\infty}$-type by the $C^{\infty}$-induction of $\psi_{\xi}$,
\[
\eta^{\infty}_{\xi}:=C^{\infty}\text{-}{\rm Ind}_{N_m}^N\psi_{\xi}.
\]
This is a weakly cyclic representation in the sense of \cite[Definition 2.2]{Y}~(see also \cite[Example 2.3]{Y}).
Let $\eta_{\xi}$ be an irreducible unitary representation of $N$ defined by the $L^2$-induction of $\psi_{\xi}$, 
\[
\eta_{\xi}:=L^2\text{-}{\rm Ind}_{N_{\m}}^N\psi_{\xi},
\]
which is realized on $L^2(\R\oplus J)\simeq L^2(N_{\frak m}\backslash N)$. This is nothing but the Schr\"odinger representation in the usual sense~(cf.~\cite[Section 1.2]{LV}). 

As is well known, the Schr\"odinger representation $\eta_{\xi}$ extends to a representation $\Omega_{\xi}$ of $Mp(W_J)\ltimes N$ defined by
\[
\Omega_{\xi}(g)=\omega_{\xi}(g_1)\otimes\eta_{\xi}(n)\quad(g=g_1n,~g_1\in Mp(W_J),~n\in N),
\]
where $\omega_{\xi}$ denotes the Weil representation of $Mp(W_J)$ with the same representation space as that of $\eta_{\xi}$. 
We now discuss unitary representations of the group given by the semi-direct product $R_J=H_J^{0,1}\ltimes N$. It should be remarked that the center of $N$ is also that of $R_J$, for which note that $H_J^{0,1}$ acts trivially on the center of $N$ by its conjugation on $N$. This leads to the notion of the central characters for representations of $R_J$ as those of $N$ can have central characters. 

We denote the restriction of $\Omega_{\xi}$ to $H_J^{0,1}\ltimes N$ also by $\Omega_{\xi}$. It should be now noted that $\Omega_{\xi}|_{H_J^{0,1}}$ is not always an ordinary representation. This can be a projective representation with a multiplier system valued in $\{\pm 1\}$.  For instance such a projective representation with the non-trivial multiplier occurs when $G$ is of type $F_4$, for which we have $H_J^{0,1}$ is the real symplectic group $Sp(3,\R)$ of degree three. Needless to say, the triviality of the multiplier system means that the representation is ordinary.
\begin{prop}\label{UnitaryDual-Jacobi}
Any unitary representation of $R_J=H_J^{0,1}\ltimes N$ with the fixed non-trivial central character parametrized by $\xi$ is a representation of the form $U_1\otimes \Omega_{\xi}$ defined by
\[
U(g):=U_1(g_1)\otimes \Omega_{\xi}(g)\quad\forall g=(g_1,n)\in R_J=H_J^{0,1}\ltimes N
\]
with an ordinary or projective unitary representation $U_1$ of $H_J^{0,1}$ whose multiplier system coincides with that of $\Omega_{\xi}|_{H_J^{0,1}}$.
Conversely, given any ordinary or projective unitary representation $U_1$ of $H_J^{0,1}$ with the same condition on the multiplier system as above, the representation of $R_J$ defined by
\[
U_1(g_1)\otimes \Omega_{\xi}(g)\quad\forall g=(g_1,n)\in H_J^{0,1}\ltimes N=R_J
\]
is an ordinary unitary representation. The unitary representation $U$ of $R_J$ is irreducible if and only if so is $U_1$.
\end{prop}
\begin{proof}
According to the Stone von-Neumann theorem~(cf.~\cite[Sections 1.3.10--1.3.13]{LV}), any unitary representation of $N$ with the non-trivial central charatcer indexed by $\xi\not=0$ is a multiple of the Schr\"odinger representation $(\eta_{\xi},V_{\xi})$, where the representation space $V_{\xi}$ of $\eta_{\xi}$ is given by $L^2(\R\oplus J)$. Here we remark that the multiplicity can be infinite.

For any unitary representation $(U,\cH)$ of $R_J$ with the central character indexed by $\xi$, let us consider its restriction to $N$. 
By the Stone von-Neumann theorem we can consider the space $\cH_1:=\Hom_N(V_{\xi},\cH)$ of unitary intertwining operators and this can be viewed as a unitary representation of $N$. 
In fact, the space $\cH_1$ forms a Hilbert space with respect to the Hilbert-Schmidt  type norm defined by $\displaystyle\frac{||\phi(x)||_{\cH}}{||x||_{\xi}}~(x\in V_{\xi}\setminus\{0\})$ for $\phi\in \cH_1$, where $||*||_{\cH}$ and $||*||_{\xi}$ denote the norms of $\cH$ and $V_{\xi}$ respectively~(or the inner product of $V_{\xi}^*\otimes\cH\simeq \Hom_N(V_{\xi},\cH)$ given by the tensor product of the inner products of $V_{\xi}^*$ and $\cH$, with the dual space $V_{\xi}^*$ of $V_{\xi}$). Then $\cH$ can be regarded as a Hilbert space tensor product $\cH=\cH_1\otimes V_{\xi}$. For $g=(g_1,n)\in R_J=H_J^{0,1}\ltimes N$ we put 
\[
U_1(g)\phi:=U(g)\cdot\phi\cdot \Omega_{\xi}(g)^{-1}\quad(\phi\in\cH_1). 
\]
We see that this depends only on $H_J^{0,1}$ in view of the Stone von-Neumann theorem, and $U_1$ is thus an ordinary unitary representation or a projective unitary representation of $H_J^{0,1}$ with the same multiplier system with that of $\Omega_{\xi}|_{H_J^{0,1}}$.

We see that
\[
U_1(g)\phi(x)=U_1(g_1)\phi(x)
\]
for $x\in V_{\xi}$ and $g\in R_J$ with its $H_J^{0,1}$-part $g_1$ and that 
\[
U(g)\phi(x)=U(g)\phi(\Omega_{\xi}(g)^{-1}\Omega_{\xi}(g)x)=U_1(g)\phi\otimes \Omega_{\xi}(g)x=U_1(g_1)\phi\otimes \Omega_{\xi}(g)x
\] 
for $x\in V_{\xi}$ and $(g,g_1)$ as above, where note that the second equality is due to the identification $\cH=\cH_1\otimes V_{\xi}$. We have therefore seen that $U(g)=U_1(g_1)\otimes \Omega_{\xi}(g)$. We further verify that if $U$ is irreducible, the irreducibility of $U_1$ has to hold since otherwise there is a contradiction to the irreducibility of $U$.

Conversely, given an unitary ordinary or projective representation $(U_1,\cH_1)$ of $H_J^{0,1}$ with the same multiplier system with $\Omega_{\xi}|_{H_J^{0,1}}$, we define a representation $(U,\cH)$ of $R_J$ by
\[
U(g):=U_1(g_1)\otimes \Omega_{\xi}(g_1n)\quad(g=(g_1,n)\in H_J^{0,1}\ltimes N=R_J),
\]
where note that $U$ is an ordinary representation of $R_J$ by the assumption on the multiplier system. 
This has the central character indexed by $\xi$. If $U_1$ is irreducible and $(U,\cH)$ has a proper non-zero $R_J$-invariant subspace $(U',\cH')$ we put $\cH'_1:=\Hom_N(V_{\xi},\cH')$ and define a representation $(U'_1,\cH'_1)$ of $H_J^{0,1}$ by
\[
U'_1(g_1)\phi(v):=U'(g_1)\cdot\phi(\Omega_{\xi}(g_1)^{-1}v)\quad(\phi\in\cH'_1,~v\in V_{\xi}).
\]
Now note that $(U'_1,\cH'_1)$ can be viewed as a subrepresentation of $(U_1,\cH_1)$ by
\[
\cH'_1\simeq V^*_{\xi}\otimes\cH'\subset V^*_{\xi}\otimes\cH\simeq\cH_1,
\]
where $V_{\xi}^*$ denotes the dual space of $V_{\xi}$. 
Therefore we have a contradiction to the irreducibility of $(U_1,\cH_1)$.
\end{proof}

\subsection{Generalized Whittaker functions}\label{GWF-definition}
As is known by \cite[Proposition 4.1]{G-W}, the maximal compact subgroup $K$ of $G$ is isomorphic to $SU(2)\times M/\langle-1,\epsilon\rangle$ with the compact real form $M$ of the complex group $M_{\C}$ as in \cite[Table 2.6]{G-W}, where $\epsilon$ denotes the unique element of order two in the center of $M$. 

For a positive integer $n$ such that $2n\ge\dim W_J$ let $\pi_{n}$ be a quaternionic discrete series representation with minimal $K$-type $(\tau_n, \bV_n)$ given by the trivial extension of the $2n$-th symmetric tensor representation of $SU(2)$ to $K$. 
In what follows, we can identify $(\tau_n,\bV_n)$ with its contragredient representation $(\tau_n^*,\bV_n^*)$ since the symmetric tensor representation of $SU(2)$ is self-dual and so is $\tau_n$. Let $\eta$ be a weakly cyclic representation of $N$ in the sense of \cite[Definition 2.2]{Y}. As is remarked in \cite[Example 2.3]{Y} this notion of representations of $N$ includes irreducible unitary representations.

For a discrete series representation $\pi$ of $G$ we are interested in the space of the intertwining operators:
\[
\Hom_{(\g,K)}(\pi,C^{\infty}\operatorname{-Ind}_N^G(\eta))
\]
as $(\g,K)$-modules, where $\g$ denotes the Lie algebra of $G$ as usual. This intertwining space is investigated by considering the restriction to the minimal $K$-type as follows:
\[
I:\Hom_{(\g,K)}(\pi,C^{\infty}\operatorname{-Ind}_N^G(\eta))\ni\Phi\mapsto \Phi\circ\iota\in\Hom_{K}(\tau,C^{\infty}\operatorname{-Ind}_N^G(\eta))
\]
with the embedding $\iota:\tau\rightarrow\pi$ as $K$-modules. 
\begin{defn}
Let $\pi$ be a quaternionic discrete series representation $\pi_n$. An element in ${\rm Im}(I)$ is called a generalized Whittaker function for $\pi_n$ with the minimal $K$-type $\tau_n$ for $\eta$. 
\end{defn}
We now introduce
\[
C^{\infty}_{\eta,\tau_n^*}(N\backslash G/K):=\{W:G\rightarrow V_{\eta}\boxtimes \bV_n^*\mid W(ngk)=\eta(n)\boxtimes\tau_n^*(k)^{-1}W(g)~\forall (n,g,k)\in N\times G\times K\}
\]
with the representation space $V_{\eta}$ of $\eta$. Let us note that there is a canonical identification
\[
\Hom_{K}(\tau_n,C^{\infty}\operatorname{-Ind}_N^G(\eta))\simeq C^{\infty}_{\eta,\tau_n^*}(N\backslash G/K)
\]
and that ${\rm Im}(I)$ can be then regarded as a subspace of this. 
Due to Yamashita \cite[Proposition 2.1]{Y} we have a characterization of $\Hom_{(\g,K)}(\pi,C^{\infty}\operatorname{-Ind}_N^G(\eta))$ as follows:
\begin{prop}
Let ${\cal D}$ be the Dirac-Schmid operator(cf.~\cite{S1},~\cite[Section 2.1]{Y}) for a (general) discrete series representation $\pi$. The restriction map above induces the following injection
\[
\Hom_{(\g,K)}(\pi,C^{\infty}\operatorname{-Ind}_N^G(\eta))\hookrightarrow \{W\in C^{\infty}_{\eta,\tau_n^*}(N\backslash G/K)\mid {\cal D}\cdot W=0\}.
\]
\end{prop}
Yamashita \cite[Theorem 2.4]{Y} showed that this is an isomorphism for general discrete series representations of a connected semisimple real Lie group with the condition that the highest weights of the minimal $K$-types of discrete series~(called Blattner parameter) are ``far from the wall''. More precisely Yamashita studied intertwining operators from discrete series representations to representations induced from weekly cyclic representations~(cf.~\cite[Definition 2.2]{Y}) of a closed subgroup of the connected semisimple real groups. 
However there seems to be an expectation that this condition can be removed. In fact, we will show by ourselves that the injection above turns out to be a bijection for quaternionic discrete series $\pi_{n}$. To verify this we essentially use Proposition \ref{WhittakerModule} together with explicit formulas for generalized Whittaker functions etc. Proposition \ref{WhittakerModule} will be used also at some points of Sections \ref{FJ-K-principle-contspec} and \ref{Application}.
\subsection{Generalized Whittaker functions for Schr\"odinger representations}\label{GWF-STrep}
It is now to start studying generalized Whittaker functions for $\eta_{\xi}^{\infty}$ in detail. We carry out this study under some  assumption on the separation of variables, which is enough for our purpose. 

We note that $\eta_{\xi}^{\infty}$ is realized on $C^{\infty}(\R\oplus J)$ in view of the identification $\R\oplus J\simeq N_m\backslash N$. The generalized Whittaker functions for $\eta_{\xi}^{\infty}$ in our concern are given as
smooth $C^{\infty}(\R\oplus J)\boxtimes\bV_n$-valued functions $W_{\xi}$ satisfying
\[
W_{\xi}(ugk)=\eta_{\xi}^{\infty}(u)\boxtimes\tau_{n}^*(k)^{-1}W_{\xi}(g)\quad(u,g,k)\in N\times G\times K,
\]
for which we note that, as has been remarked, $\tau_n$ is self-dual, namely $\tau_n^*\simeq\tau_n$ and we can identify $\bV_n^*$ with $\bV_n$ for the contragredient $\tau_n^*$ of $\tau_n$. 
We write $W_{\xi}$ as 
\[
W_{\xi}(g)=\sum_{v=-n}^{n}\phi_v(g)\left(\frac{x^{n+v}}{(n+v)!}\right)\left(\frac{y^{n-v}}{(n-v)!}\right)
\]
with $C^{\infty}(\R\oplus J)$-valued smooth function $\phi_v$, where $\{\left(\frac{x^{n+v}}{(n+v)!}\right)\left(\frac{y^{n-v}}{(n-v)!}\right)\mid -n\le v\le n\}$ denotes the basis of $\bV_n$ as in \cite[Section 7.3]{P1} and \cite{P2}. 

We now introduce the new coordinate $(a,b):=(a,b,0,0)\in W_J$ for $a\in\R,~b\in J$. In addition to this we remark that, using the group coordinate $n(X),~M_Y:=M(N_J(Y),U_Y)\in H_J^{0,1}$~(cf.~\cite[Section 2.2]{P1}) with $(X,Y)\in J\times\Omega$ and $U_Y$ defined in \cite[Section 2.1]{P1}, we have 
\[
Z=n(X)M_Y\cdot\sqrt{-1}1_J=X+\sqrt{-1}Y\quad(\text{for $Z\in D_J$})
\]
(see around the end of the proof for \cite[Proposition 2.3.1]{P1}), for which recall that $H_J^0$ acts on $D_J$ transitively as is already remarked in Section \ref{Cubicforms-Groups}~(replacing $H_J^0$ by $H_J^{0,1}$). 
The group $H_J^0$ is a product of its semisimple part $H_J^{0,1}$ and the connected torus isomorphic to $\R_{>0}$ contained in the center of $H_J^0$. We write a general element of this torus group by the coordinate $w\in\R_{>0}$. Taking an Iwasewa decomposition of $G$~(cf.~\cite[Theorem 5.12]{Kn}) into account we then see that the functions $W_{\xi}$ and $\phi_v$ are determined by their restriction to
\[
(a,b)\times n(X)M_Y\times w\in (\R\oplus J)\times H_J^{0,1}\times\R_{>0}
\]
in view of the left $N_m$-equivariance by $\psi_{\xi}$ and the right $K$-equivariance by $\tau_n^*$~(see also \cite[Section 7.5]{P1}). Here we note that $n(X)M_Y$ is identified with $Z=X+\sqrt{-1}Y$ since $n(X)M_Y\cdot\sqrt{-1}1_J=X+\sqrt{-1}Y$ as we have seen.

We write down the characterizing differential equations of $W_{\xi}$ based on \cite[Theorem 7.3.1]{P1} by rewriting it in the setting for $\eta=\eta_{\xi}^{\infty}$. 
To this end we define the differential operator $D_x^{(a,b)}$ by the partial derivative of the coordinate $(a,b)$ in the $x$-direction for $x\in \R\oplus J$. This can be extended to the complex variable in $b$ by $D_{(a_0,b_{0,x}+\sqrt{-1}b_{0,y})}^{a,b}:=D_{(a_0,b_{0,x})}^{a,b}+\sqrt{-1}D_{(a_0,b_{0,y})}^{a,b}$ for $a_0\in\R$ and $b_{0,x}+\sqrt{-1}b_{0,y}\in J\otimes\C$. We introduce 
\begin{align*}
&\tilde{Z}:=(1,-Z,Z^{\sharp},-N_J(Z))\in W_J\otimes\C,\\ 
&MV(E):=\frac{\tr(E)}{2}\tilde{Z}+(0,E_Y,-Z\times E_Y,(Z^{\sharp},E_Y))\in W_J\otimes\C
\end{align*}
for $Z=X+\sqrt{-1}Y\in\cH_J$ and $E\in J$, where $\tr(E):=(E,1_J)$ and $E_Y:=U_{Y^{1/2}}(E)$. Here note that the notation $MV(E)$ is different from that introduced at \cite[(26)]{P1}. 
We furthermore put
\[
\tilde{Z}_{a,b}:=(1,-Z)\in(\R\oplus J)\otimes\C,\quad MV(E)_{a,b}=\frac{\tr(E)}{2}(1,-Z)+(0,E_Y)\in(\R\oplus J)\otimes\C.
\]
By $\tilde{Z}^*$, $MV(E)^*$, $\tilde{Z}_{a,b}^*$ and $MV(E)_{a,b}^*$, we denote the notations defined by the complex conjugate $\bar{Z}$ instead of $Z$ in the corresponding three notations. In addition to these we use the notation $D_{Z(E)},~D_{Z^*(E)}$ with $E\in J$, following \cite[Section 7.1,~(3),~(4)]{P1}. These are viewed as the Cauchy Riemann operators or its complex conjugates for $D_J$.  
With these notations above the differential equations are given as follows:
\begin{prop}\label{Diff-eq_general}
\begin{enumerate}
\item For $-(n-1)\le v\le n$ we have 
\begin{equation}\label{Diff-eq_gen1st}
(w\partial w-2(n+1)+v+4\pi w^2\xi)\phi_v=\sqrt{-1}wN_J(Y)^{-1/2}(D^{(a,b)}_{\tilde{Z}^*_{a,b}}+2\pi\sqrt{-1}\xi\langle (a,b),\tilde{Z}^*\rangle)\phi_{v-1},
\end{equation}
\item For $-n\le v\le n-1$ we have
\begin{equation}\label{Diff-eq_gen2nd}
(w\partial w-2(n+1)-v-4\pi w^2\xi)\phi_v=\sqrt{-1}wN_J(Y)^{-1/2}(D^{(a,b)}_{\tilde{Z}_{a,b}}+2\pi\sqrt{-1}\xi\langle (a,b),\tilde{Z}\rangle)\phi_{v+1}.
\end{equation}
\item For $-n\le v\le n-1$ and for any $E\in J$ we have
\begin{equation}\label{Diff-eq_gen3rd}
(D_{Z(E)}+\frac{v}{2}\tr(E))\phi_v=-\sqrt{-1}wN_J(Y)^{-1/2}(D^{(a,b)}_{MV(E)_{a,b}}+2\pi\sqrt{-1}\xi\langle(a,b),MV(E)\rangle)\phi_{v+1}.
\end{equation}
\item For $-(n-1)\le v\le n$ and for any $E\in J$ we have
\begin{equation}\label{Diff-eq_gen4th}
(D_{Z^*(E)}-\frac{v}{2}\tr(E))\phi_v=-\sqrt{-1}wN_J(Y)^{-1/2}(D^{(a,b)}_{MV(E)_{a,b}^*}+2\pi\sqrt{-1}\xi\langle(a,b),MV(E)^*\rangle)\phi_{v-1}.
\end{equation}
\end{enumerate}
\end{prop}
\subsection{Reduction of the differential equations~(Proof of Proposition \ref{Diff-eq_general-2})}\label{Pf-Prop2.4}
We put $\phi_v:=w^{2(n+1)}G_v$. The system of the differential equations in Proposition \ref{Diff-eq_general} is changed for $G_v$ a little. More precisely, \lq\lq$-2(n+1)$\rq\rq is removed from (\ref{Diff-eq_gen1st}) and (\ref{Diff-eq_gen2nd}). 
Hereafter we mean these two modified differential equations by $(\ref{Diff-eq_gen1st})$ and $(\ref{Diff-eq_gen2nd})$. From now on, we make the assumption on the generalized Whittaker functions $W_{\xi}$ as follows:
\begin{asm}
Each function $G_v$ satisfies the separation of the variables with respect to $w$ and the other variables, meaning that $G_v(a,b,X,Y,w)=F_v(a,b,X,Y)H_v(w)$. 
\end{asm}
The generalized Whittaker functions with this assumption are enough for our purpose
\begin{prop}\label{Diff-eq_general-2}
Let $G_v$ be as above. We have $G_v\equiv 0$ for $-(n-1)\le v\le n-1$. 
The differential equations in Proposition \ref{Diff-eq_general} are therefore reduced to the following:\\
(1)~The differential equations (\ref{Diff-eq_gen1st}) in Proposition \ref{Diff-eq_general} are reduced to
\begin{align}
&(w\partial w+n+4\pi w^2\xi)G_{n}=0,\label{Diff-eq_gen1st-1}\\
&(D^{(a,b)}_{\tilde{Z}^*_{a,b}}+2\pi\sqrt{-1}\xi\langle (a,b),\tilde{Z}^*\rangle)G_{-n}.\label{Diff-eq_gen1st-2}
\end{align}
(2)~The differential equations (\ref{Diff-eq_gen2nd}) in Proposition \ref{Diff-eq_general} are reduced to
\begin{align}
&(w\partial w+n-4\pi w^2\xi)G_{-n}=0,\label{Diff-eq_gen2nd-1}\\
&(D^{(a,b)}_{\tilde{Z}_{a,b}}+2\pi\sqrt{-1}\xi\langle (a,b),\tilde{Z}^\rangle)G_{n}.\label{Diff-eq_gen2nd-2}
\end{align}
(3)~The differential equations (\ref{Diff-eq_gen3rd}) in Proposition \ref{Diff-eq_general} are reduced to
\begin{align}
&(D_{Z(E)}-\frac{n}{2}\tr(E))G_{-n}=0,\label{Diff-eq_gen3rd-1}\\
&(D^{(a,b)}_{MV(E)_{a,b}}+2\pi\sqrt{-1}\xi\langle(a,b),MV(E)\rangle)G_{n}=0.\label{Diff-eq_gen3rd-2}
\end{align}
(4)~The differential equations (\ref{Diff-eq_gen4th}) in Proposition \ref{Diff-eq_general} are reduced to
\begin{align}
&(D_{Z^*(E)}-\frac{n}{2}\tr(E))G_{n}=0,\label{Diff-eq_gen4th-1}\\
&(D^{(a,b)}_{MV(E)_{a,b}^*}+2\pi\sqrt{-1}\xi\langle(a,b),MV(E)^*\rangle)G_{-n}=0.\label{Diff-eq_gen4th-2}
\end{align}
\end{prop}
This proposition leads to an important consequence on the generalized Whittaker functions $W_{\xi}$ as we will see in Section \ref{Results-GWF-FJS}. The proof proceeds with several lemmas.
\begin{lem}\label{degtwo-eqs_general}
For $-(n-1)\le v\le n$,
\begin{align}
&(w\partial w-v-4\pi w^2\xi)(w\partial w+v+4\pi w^2\xi)G_v=\nonumber\\
&-w^2N_J(Y)^{-1}(D^{(a,b)}_{\tilde{Z}^*_{a,b}}+2\pi\sqrt{-1}\xi\langle (a,b),\tilde{Z}^*\rangle)(D^{(a,b)}_{\tilde{Z}_{a,b}}+2\pi\sqrt{-1}\xi\langle (a,b),\tilde{Z}\rangle)G_v\label{degtwo-eq_1st}
\end{align}
and for $-n\le v\le n-1$,
\begin{align}
&(w\partial w+v+4\pi w^2\xi)(w\partial w-v-4\pi w^2\xi)G_v=\nonumber\\
&-w^2N_J(Y)^{-1}(D^{(a,b)}_{\tilde{Z}_{a,b}}+2\pi\sqrt{-1}\xi\langle (a,b),\tilde{Z}\rangle)(D^{(a,b)}_{\tilde{Z}^*_{a,b}}+2\pi\sqrt{-1}\xi\langle (a,b),\tilde{Z}^*\rangle)G_v.\label{degtwo-eq_2nd}
\end{align}
These two are equivalent to each other for $v$ with $-(n-1)\le v\le n-1$.
\end{lem}
\begin{proof}
The two differential equations follow from (\ref{Diff-eq_gen1st}) and (\ref{Diff-eq_gen2nd}). To check the last assertion, calculate the difference of (\ref{degtwo-eq_1st}) and (\ref{degtwo-eq_2nd}). Then those of both sides are $16\pi w^2\xi$. 
In fact, the difference of the left hand side is easy to calculate. On the other hand, that of the right hand side is  
\begin{align*}
&-w^2N_J(Y)^{-1}(D^{(a,b)}_{\tilde{Z}^*_{(a,b)}}\cdot2\pi\sqrt{-1}\xi\langle(a,b),\tilde{Z}\rangle- D^{(a,b)}_{\tilde{Z}_{(a,b)}}\cdot2\pi\sqrt{-1}\xi\langle(a,b),\tilde{Z}^*\rangle)\\
&=2\pi\sqrt{-1}\xi w^2N_J(Y)^{-1}(N_J(Z)-(Z^{\sharp},\bar{Z})+(Z,\bar{Z}^{\sharp})-N_J(\bar{Z}))\\
&=2\pi\sqrt{-1}\xi N_J(Y)^{-1}N_J(Z-\bar{Z})w^2=16\pi(\sqrt{-1})^4w^2\xi=16\pi w^2\xi,
\end{align*}
which is settled by reviewing the definition of  $D^{(a,b)}_{\tilde{Z}_{(a,b)}},~D^{(a,b)}_{\tilde{Z}^*_{(a,b)}}$ and noting the fundamental property of $N_J$ enumerated as $4$~(cf.~Section \ref{Cubicforms-Groups}). 
This implies the equivalence of (\ref{degtwo-eq_1st}) and (\ref{degtwo-eq_2nd}).
\end{proof}
The following lemma is due to \cite[p1259,~(28)]{P1}. We will often use it in the coming argument. 
\begin{lem}\label{derivative-cubicpoly}
For $E\in J$
\[
D_{Z(E)}(N_J(Y))=N_J(Y)\tr(E),\quad D_{Z^*(E)}(N_J(Y))=N_J(Y)\tr(E)
\]
hold.
\end{lem}
The next lemma is useful to determine the ``$F_v$-parts'' of generalized Whittaker functions explicitly, where see the assumption of the separation of variables for the ``$F_v$-parts''. 
\begin{lem}\label{Diffeq-rk1-Lem}
\begin{enumerate}
\item We have
\begin{align}
D_{\tilde{Z}_{a,b}}^{(a,b)}\cdot(a^2N_J(Z)+a(b,Z^{\sharp})+\frac{1}{2}(b,b\times Z))=\langle(a,b),\tilde{Z}\rangle,\label{Diffeq-rk1-Lem1st}\\
D_{\tilde{Z}^*_{a,b}}^{(a,b)}\cdot(a^2N_J(\bar{Z})+a(b,\bar{Z}^{\sharp})+\frac{1}{2}(b,b\times \bar{Z}))=\langle(a,b),\tilde{Z}^*\rangle\label{Diffeq-rk1-Lem2nd}.
\end{align}
\item For $E\in J$ we have
\begin{align}
D_{MV(E)_{a,b}}^{(a,b)}\cdot(a^2N_J(Z)+a(b,Z^{\sharp})+\frac{1}{2}(b,b\times Z))=\langle(a,b),MV(E)\rangle,\label{Diffeq-rk1-Lem3rd}\\
D_{MV(E)_{a,b}^*}^{(a,b)}\cdot(a^2N_J(\bar{Z})+a(b,\bar{Z}^{\sharp})+\frac{1}{2}(b,b\times \bar{Z}))=\langle(a,b),MV(E)^*\rangle\label{Diffeq-rk1-Lem4th}.
\end{align}
\end{enumerate}
\end{lem}
\begin{proof}
We only check (\ref{Diffeq-rk1-Lem1st}) and (\ref{Diffeq-rk1-Lem3rd}) since (\ref{Diffeq-rk1-Lem2nd}) and (\ref{Diffeq-rk1-Lem4th}) are settled similarly. 
Let $f_{Z}(a,b):=a^2N_J(Z)+a(b,Z^{\sharp})+\frac{1}{2}(b,b\times Z)$. Based on the cubic norm structure explained in Section \ref{Cubicforms-Groups} we verify
\begin{align*}
&\underset{t=0}{\lim}\frac{f_Z(a+t,b)-f_Z(a,b)}{t}=2aN_J(Z)+(b,Z^{\sharp}),\\
&\underset{t=0}{\lim}\frac{f_Z(a,b-tZ)-f_Z(a,b)}{t}=-3aN_J(Z)-2(b,Z^{\sharp})
\end{align*}
by direct calculation. To be precise, in the course of the calculation, we need $(Z,b\times Z)=(b,Z\times Z)=2(b,Z^{\sharp})$, for which we note that the first equality is verified by the sharp symmetry of the trace paring $(*,*)$~(cf.~\cite[p.191]{Mc}). 
These yield (\ref{Diffeq-rk1-Lem1st}) since $\langle(a,b),\tilde{Z}\rangle=-aN_J(Z)-(b,Z^{\sharp})$.

As for (\ref{Diffeq-rk1-Lem3rd}) we remark that $MV(E)$ is the sum of $\displaystyle\frac{\tr(E)}{2}\tilde{Z}$ and $(0,E_Y,-Z\times E_Y,(Z^{\sharp},E_Y))$. 
Reviewing the definition of $D_{MV(E)_{a,b}}^{(a,b)}$ we thereby see that the proof of (\ref{Diffeq-rk1-Lem3rd}) is reduced to the calculation of $\underset{t=0}{\lim}\displaystyle\frac{f_Z(a,b+tE)-f_Z(a,b)}{t}$ for $E\in J$, which is verified to be
\[
a(E,Z^{\sharp})+\frac{1}{2}(b,E\times Z)+\frac{1}{2}(E,b\times Z)=a(E,Z^{\sharp})+(b,Z\times E)=\langle(a,b),(0,E,-Z\times E,(Z^{\sharp},E))\rangle
\]
based on the cubic form structure~(cf.~Section \ref{Cubicforms-Groups}). 
For this we note that the aforementioned sharp symmetry of the trace paring is used to verify the first equality. This settles the formula (\ref{Diffeq-rk1-Lem3rd}).
\end{proof}
In what follows, we frequently use the following lemma.
\begin{lem}\label{Prop5-3_1stLem}
Let $F_v$ and $H_v$ be as in the assumption of the separation of variables for $G_v$. For $-(n-1)\le v\le n-1$ we have the following:
\begin{enumerate}
\item Assume $(D_{Z(E)}+\displaystyle\frac{\tr(E)}{2}v)F_v\not=0$ for some $E\in J$. Then, as a moderate growth solution of $H_v(w)$, we obtain
\[
H_v(w)=
\begin{cases}
w^{-v}\exp(-2\pi w^2\xi)&(\xi>0),\\
w^{-v}\exp(2\pi w^2\xi)&(\xi<0),
\end{cases}
\]
up to scalars.
\item Assume $(D_{Z^*(E)}-\displaystyle\frac{\tr(E)}{2}v)F_v\not=0$ for some $E\in J$. Then, as a moderate growth solution of $H_v(w)$, we obtain
\[
H_v(w)=
\begin{cases}
w^v\exp(-2\pi w^2\xi)&(\xi>0),\\
w^v\exp(2\pi w^2\xi)&(\xi<0),
\end{cases}
\]
up to scalars.
\end{enumerate}
\end{lem}
\begin{proof}
We think only of the case $\xi<0$ since that of $\xi>0$ is settled similarly. Under the assumptions of the parts 1 and 2, apply the differential operators $w\partial w+v+4\pi w^2\xi$ and $w\partial w-v-4\pi w^2\xi$ on the equations (\ref{Diff-eq_gen3rd}) and (\ref{Diff-eq_gen4th}) for $v$ in Proposition \ref{Diff-eq_general} respectively. For the calculation of the infinitesimal action of these operators,  we note the equations (\ref{Diff-eq_gen1st}) and (\ref{Diff-eq_gen2nd}). We then have two differential equations:
\begin{align*}
&(w\partial w+v+4\pi w^2\xi)H_v=C_v w^2H_v,\\
&(w\partial w-v-4\pi w^2\xi)H_v=C'_v w^2H_v,
\end{align*}
with constants $C_v,~C'_v$ independent of any variables. Here note that we use the assumption of the separation of variables and that we need $(D_{Z(E)}+\displaystyle\frac{v}{2}\tr(E))F_v\not=0$~(respectively~$(D_{Z^*(E)}-\displaystyle\frac{v}{2}\tr(E))F_v\not=0)$ to get the first equation~(respectively~the second equation). 
For this proof we mainly give the argument to solve the second equation. We have a solution
\[
H_v(w)=w^v\exp((2\pi\xi+\frac{C'_v}{2})w^2),
\]
up to constant multiples independent of $w$. On the other hand, in view of (\ref{degtwo-eq_2nd}) in Lemma \ref{degtwo-eqs_general} and the assumption of the separation of variables, $H_v$ also satisfies another differential equation
\[
(w\partial w+v+4\pi w^2\xi)(w\partial w-v-4\pi w^2\xi)H_v={C'}^0_v w^2 H_v
\]
with a constant ${C'}^0_v$ independent of any variables. Apply this equation to the solution above and we see
\[
C'_v={C'}^0_v=0\quad\text{or}\quad C'_v=-8\pi\xi,~{C'}^0_v=-16\pi\xi(v+1).
\]
Since $H_v$ is of moderate growth by assumption we see for a negative $\xi$ that the former is true, which means 
\begin{equation}\label{gen-solution-H_v}
H_v(w)=w^{v}\exp(2\pi w^2\xi),
\end{equation}
up to scalars independent of any variables. 
Similarly from the first differential equation around the beginning of this proof we obtain
\begin{equation}\label{gen-solution-H_v2}
H_v(w)=w^{-v}\exp(2\pi w^2\xi)
\end{equation}
up to scalar independent of any variables. As a result the two assertions are verified.
\end{proof}
\begin{lem}\label{Prop5-3_2ndLem}
Let $v\not=0$ with $-(n-1)\le v\le n-1$.
\begin{enumerate}
\item Assume that
\[
(D_{Z(E)}+\frac{v}{2}\tr(E))F_v=(D_{Z^*(E)}-\frac{v}{2}\tr(E))F_{v}=0
\]
for any $E\in J$. 
Then $G_v\equiv 0$ for $v$ above.
\item Assume that
\[
(D_{Z(E)}+\frac{v}{2}\tr(E))F_v\not=0,\quad (D_{Z^*(E)}-\frac{v}{2}\tr(E))F_{v}\not=0
\]
for some $E\in J$ and that $G_v$ is of moderate growth. 
Then $G_v\equiv 0$ for $v$ above.
\end{enumerate}
\end{lem}
\begin{proof}
We first verify the 1st assertion. Its assumption and Lemma \ref{derivative-cubicpoly} imply
\[
F_v=N_J(Y)^{-v/2}\Phi_v(Z)=N_J(Y)^{v/2}\Psi_v(Z)
\]
with a holomorphic $\Phi$ and an anti-holomorphic $\Psi$. For this we note that $D_{Z(E)}$ and $D_{Z^*(E)}$ are viewed as Cauchy Riemann operator or its complex conjugate respectively as we have remarked before Proposition \ref{Diff-eq_general}. However, this equality is impossible if $F_v\not\equiv 0$.

The second assertion follows from Lemma \ref{Prop5-3_1stLem} since the two solutions in this lemma do not compatible with each other if $H_v(w)\not\equiv 0$,  for which we should note that the assumption $v\not=0$ is necessary to justify this argument.
\end{proof}
\begin{lem}\label{Prop5-3_3rdLem}
Suppose that $G_v$ is of moderate growth with respect to $w$ for $-(n-1)\le v\le n-1$.
\begin{enumerate}
\item Suppose that $\xi>0$. Let $v$ be an integer in $-(n-2)\le v\le n-1$. Assume 
$(D_{Z(E)}+\displaystyle\frac{v}{2}\tr(E))F_v\not=0$
for some $E\in J$. Then $G_{v-1}\equiv 0$ for $v\not=0$.
\item Suppose that $\xi<0$. 
Let $v$ be an integer in $-(n-1)\le v\le n-2$. Assume  
$(D_{Z^*(E)}-\displaystyle\frac{v}{2}\tr(E))F_v\not=0$
for some $E\in J$. Then $G_{v+1}\equiv 0$ for $v\not=0$.
\end{enumerate}
\end{lem}
\begin{proof}
We verify the second assertion in a detailed manner. For the first assertion we give some rough explanation around the end of the proof since it is settled by a similar manner. 

According to Lemma \ref{Prop5-3_1stLem} part 2, we get  
\begin{equation}\label{gen-solution-H_v11}
H_v(w)=w^v\exp(2\pi w^2\xi)
\end{equation}
up to scalars. 
We begin with verifying that we can exclude the following case:
\[
(D_{Z^*(E)}-\displaystyle\frac{\tr(E)}{2}(v+1))F_{v+1}=0\quad\forall E\in J.
\]
Let us assume this. If $F_{v+1}\not\equiv 0$ and $v+1\not=0$, the part 1 of Lemma \ref{Prop5-3_2ndLem} implies $(D_{Z(E)}+\displaystyle\frac{\tr(E)}{2}(v+1))F_{v+1}\not=0$ for some $E\in J$. 

We should be careful when $v+1=0$, for which the part 1 of Lemma \ref{Prop5-3_2ndLem} is not useful. However this still holds even when $v+1=0$.  
Let us think of this case, i. e. $v=-1$. Under $(D_{Z^*(E)}-\displaystyle\frac{\tr(E)}{2}(v+1))F_{v+1}=0$~($\forall E\in J$), we firstly consider the following case: 
\[
(D_{Z(E)}+\displaystyle\frac{v}{2}\tr(E))F_v=0\quad\forall E\in J.
\]
Noting $H_v(w)$ as above and Lemma \ref{Diffeq-rk1-Lem}, we see that the equations (\ref{Diff-eq_gen2nd}),~(\ref{Diff-eq_gen3rd}) for $v=-1$ 
in Proposition \ref{Diff-eq_general} then imply that $F_{v+1}=F_0$ equals to
\begin{equation}\label{solution-F_v+1}
\exp(-2\pi\sqrt{-1}\xi(a^2N_J(Z)++a(b,Z^{\sharp})+\frac{1}{2}(b,b\times Z)))C(X,Y),
\end{equation}
with a function $C(X,Y)$ in $(X,Y)$ independent of $(a,b)$. When $v=-1$, if $F_{v+1}\equiv F_0\not\equiv 0$, it is impossible that $(D_{Z(E)}+\displaystyle\frac{\tr(E)}{2}(v+1))F_{v+1}=0$ for all $E\in J$ since, if this holds, $C(X,Y)$ is a holomorphic function in $Z$ while $F_0$ is also anti-holomorphic in $Z$, namely $F_0$ should be a non-zero constant in $Z$, which is impossible by (\ref{solution-F_v+1}). Therefore we deny that $(D_{Z(E)}+\displaystyle\frac{v}{2}\tr(E))F_v=0$ for any $E\in J$. We then secondly consider the following case: 
\[
(D_{Z(E')}+\displaystyle\frac{v}{2}\tr(E'))F_v\not=0\quad\exists E'\in J.
\] 
Lemma \ref{Prop5-3_1stLem} yields $H_v(w)=w^{-v}\exp(2\pi w^2\xi)$ up to scalar, which contradicts to (\ref{gen-solution-H_v11}) under the assumption $v\not=0$~(in fact, $v=-1$ for now).

We therefore see that $(D_{Z(E)}+\displaystyle\frac{\tr(E)}{2}(v+1))F_{v+1}\not=0$ has to hold for some $E\in J$~(even when $v=-1$). However we can deny this. In fact, if this holds, Lemma \ref{Prop5-3_1stLem}
leads to $H_{v+1}(w)=w^{-v-1}\exp(2\pi w^2\xi)$ up to scalars. This solution together with $H_v$ as in (\ref{gen-solution-H_v11}) contradicts to the equation (\ref{Diff-eq_gen1st}) for $v+1$. For this we also use the assumption $v\not=0$. 

As a result we have to think of the remaining case 
\[
(D_{Z^*(E)}-\displaystyle\frac{\tr(E)}{2}(v+1))F_{v+1}\not=0\quad \exists E\in J.
\]
By Lemma \ref{Prop5-3_1stLem} we have $H_{v+1}(w)=w^{v+1}\exp(2\pi w^2\xi)$ up to scalars. With this solution and $H_v$ above we then also deduce a contradiction from the equation (\ref{Diff-eq_gen1st}) for $v+1$ if $G_{v+1}\not\equiv 0$, namely we have $G_{v+1}\equiv 0$.

Now let us explain how to verify the first assertion of the proposition briefly. We begin with
\[
H_v(w)=w^{-v}\exp(-2\pi w^2\xi)
\]
up to scalar, which is due to Lemma \ref{Prop5-3_1stLem} part 1. Taking this into account the proof can be carried out for the two cases:
\begin{itemize}
\item $(D_{Z(E)}+\displaystyle\frac{\tr(E)}{2}(v-1))F_{v-1}=0$ for any $E\in J$, 
\item $(D_{Z(E)}+\displaystyle\frac{\tr(E)}{2}(v-1))F_{v-1}\not=0$ for some $E\in J$. 
\end{itemize}
For the first case it is verified that if we assume $F_{v-1}\not=0$, $(D_{Z^*(E)}-\displaystyle\frac{\tr(E)}{2}(v-1))F_{v-1}\not=0$ has to hold for some $E\in J$, which can be proved also for $v-1=0, i.e. v=1$ with the help of equations (\ref{Diff-eq_gen1st}),~(\ref{Diff-eq_gen4th}) and Lemma \ref{Prop5-3_1stLem} etc, in a manner similar to the second  assertion of the proposition. 
For both cases we deduce a contradiction from comparison of $H_v$ and $H_{v-1}$. 
For this comparison we note that we use the equation (\ref{Diff-eq_gen2nd}) with $v-1$ regarding the second case, which leads us to $G_{v-1}\equiv 0$.

We have proved the assertion consequently.
\end{proof}

\subsection*{The remaining steps for the proof of Proposition \ref{Diff-eq_general-2}
}
We write down the complete proof only for $\xi<0$ since the case of $\xi>0$ is settled quite similarly. Around the end of the proof we make some brief explanation of the proof for the case of $\xi>0$.  
We complete the proof for the proposition with the lemmas above. 

Let us assume 
\[
(D_{Z^*(E)}-\displaystyle\frac{(-(n-1))}{2}\tr(E))F_{-(n-1)}\not=0
\]
for some $E\in J$. 
By Lemma \ref{Prop5-3_3rdLem}, part 2 we see $G_{-(n-2)}\equiv 0$. For this note that $-(n-1)\not=0$ since $n$ is at least $2$, which is possible when $G$ is of type $G_2$. We then deduce $G_v\equiv 0$ for $v\ge-(n-2)$ inductively from the equation (\ref{Diff-eq_gen1st}) since we see that $H_v$ is not of moderate growth for $v>-(n-2)$ if $H_v\not\equiv 0$ for such $v$. From $(D_{Z^*(E)}-\displaystyle\frac{(-(n-1))}{2}\tr(E))F_{-(n-1)}\not=0$ and Lemma \ref{Prop5-3_1stLem}, part 2, we have $H_{-(n-1)}(w)=w^{-(n-1)}\exp(2\pi w^2\xi)$. By what we have proved so far the differential equations (\ref{Diff-eq_gen1st}),~(\ref{Diff-eq_gen4th}) for $v=-(n-2)$ imply that $G_{-(n-1)}$ equals to
\[
w^{-(n-1)}\exp(2\pi w^2\xi)\exp(-2\pi\sqrt{-1}\xi(a^2N_J(\bar{Z})++a(b,\bar{Z}^{\sharp})+\frac{1}{2}(b,b\times \bar{Z}))),
\]
up to multiplication by a function independent of $a,~b$ and $w$. This and (\ref{Diff-eq_gen3rd}) for $v=-n$ imply that $H_{-n}(w)=w^{2-n}\exp(2\pi w^2\xi)$ up to scalars, for which we may assume that $G_{-n}\not\equiv0$ since otherwise we have $G_v\equiv 0$ for $v\ge -(n-1)$ inductively by (\ref{Diff-eq_gen1st}). 
However, if $G_{-(n-1)}\not\equiv 0$, this $H_{-n}$ and $H_{-(n-1)}$ above do not satisfy the equation (\ref{Diff-eq_gen1st}) for $v=-(n-1)$, which implies $G_{-(n-1)}\equiv 0$. We thus obtain 
$G_v\equiv 0$ for $-(n-1)\le v\le n-1$.
%

We are left with the following case:
\[
(D_{Z^*(E)}-\displaystyle\frac{(-(n-1))}{2}\tr(E))F_{-(n-1)}=0\quad\forall E\in J.
\]
Suppose that $G_{-(n-1)}\not\equiv 0$. 
Taking Lemma \ref{derivative-cubicpoly} into consideration, we then note
\begin{equation}\label{AntiHol-F_{-(n-1)}}
F_{-(n-1)}=N_J(Y)^{-(n-1)/2}C_{-(n-1)}((Z,a,b))
\end{equation}
with a non-zero function $C_{-(n-1)}((Z,a,b))$ anti-holomorphic in $Z$. 
We then see that (\ref{AntiHol-F_{-(n-1)}}) is not compatible with the equation (\ref{degtwo-eq_2nd}) in terms of the powers of $N_J(Y)$.  
In fact, let us put $B_{-(n-1)}((Z,a,b)):=(D^{(a,b)}_{\tilde{Z}^*_{a,b}}+2\pi\sqrt{-1}\xi\langle (a,b),\tilde{Z}^*\rangle)C_{-(n-1)}((Z,a,b))$, which is still anti-holomorphic in $Z$. We then have 
\begin{align*}
&(D^{(a,b)}_{\tilde{Z}_{a,b}}+2\pi\sqrt{-1}\xi\langle (a,b),\tilde{Z}\rangle)(D^{(a,b)}_{\tilde{Z}^*_{a,b}}+2\pi\sqrt{-1}\xi\langle (a,b),\tilde{Z}^*\rangle)C_{-(n-1)}((Z,a,b))\\
=&(D^{(a,b)}_{\tilde{Z}_{a,b}}+2\pi\sqrt{-1}\xi\langle (a,b),\tilde{Z}\rangle)B_{-(n-1)}((Z,a,b)).
\end{align*} 
Now note that the generalized Whittaker function is real analytic and $B_{-(n-1)}$ thus admits a power series expansion with respect to $(a,b)$ with anti-holomorphic coefficient functions for all the terms. We further note that the dependence on $Z$  for the differential operator $D^{(a,b)}_{\tilde{Z}_{a,b}}+2\pi\sqrt{-1}\xi\langle (a,b),\tilde{Z}\rangle$ is holomorphic.
From the right hand side of (\ref{degtwo-eq_2nd}) we can not thereby factor out any power of $N_J(Y)$ other than $N_J(Y)^{-(n-1)/2}$.
We therefore see $G_{-(n-1)}\equiv 0$ and deduce $G_v\equiv 0$ for $v\ge -(n-1)$ inductively by verifying that $H_v$ cannot be non-zero functions of moderate growth for $v\ge-(n-1)$. 

As for the case of $\xi>0$, the proof is settled by dividing the argument into the two cases as follows:
\begin{itemize}
\item $(D_{Z(E)}+\displaystyle\frac{n-1}{2}\tr(E))F_{n-1}\not=0$ for some $E\in J$,
\item $(D_{Z(E)}+\displaystyle\frac{n-1}{2}\tr(E))F_{n-1}=0$ for any $E\in J$.
\end{itemize}
The first case begins with showing $G_{n-2}\equiv 0$ by Lemma \ref{Prop5-3_3rdLem}, part 1. This leads to $G_v\equiv 0$ for $v\le n-2$ inductively by the equation (\ref{Diff-eq_gen2nd}). In a manner similar to the case of $\xi<0$ we deduce an explicit formula for $H_{n-1}$~(respectively~$G_{n-1}$) from Lemma \ref{Prop5-3_1stLem}, part 1~(respectively~the equations (\ref{Diff-eq_gen2nd}),~(\ref{Diff-eq_gen3rd})). This together with (\ref{Diff-eq_gen4th}) for $v=n$ leads to an explicit formula for $H_n$ and to the incompatibility of $H_{n-1}$ with $H_n$ by (\ref{Diff-eq_gen2nd}) for $v=n-1$, which shows $G_{n-1}\equiv 0$.

The second case starts with getting the formula $F_{n-1}=N_J(Y)^{(n-1)/2}C_{n-1}((Z,a,b))$ with a holomorphic function $C_{n-1}((Z,a,b))$ in $Z$ like $F_{-(n-1)}$ in the case of $\xi<0$. 
Similarly we then deduce a contradiction from the comparison of powers of $N_J(Y)$ in view of (\ref{degtwo-eq_1st}). 
We therefore obtain $G_{v-1}\equiv 0$ from Lemma \ref{Prop5-3_2ndLem}.
Anyway we have therefore proved  
\[
G_v\equiv 0\quad(-(n-1)\le v\le n-1).
\]
\subsection{Results on the generalized Whittaker functions and Fourier-Jacobi models}\label{Results-GWF-FJS}
We firstly state the following result on an explicit formula for the generalized Whittaker functions, which is deduced from Proposition \ref{Diff-eq_general}.
\begin{thm}\label{GenWhittakerFct-gen}
Keep the assumption as in Proposition \ref{Diff-eq_general}. Up to scalar multiplication $C$ we have an explicit formula for $\phi_v$ of moderate growth with respect to $w$ as follows:\\
\begin{enumerate}
\item Suppose $\xi>0$. We have $\phi_v\equiv 0$ for $-n\le v\le n-1$ and 
\[
\phi_n=CN_J(Y)^{\frac{n}{2}}w^{n+2}\exp(-2\pi\xi w^2)\exp(-2\pi\xi \sqrt{-1}(a^2N_J(Z)+a(b,Z^{\sharp})+\displaystyle\frac{1}{2}(b,b\times Z))).
\]
\item Suppose $\xi<0$. We have $\phi_v\equiv 0$ for $-(n-1)\le v\le n$ and 
\[
\phi_{-n}=
CN_J(Y)^{\frac{n}{2}}w^{n+2}\exp(2\pi\xi w^2)\exp(-2\pi\xi \sqrt{-1}(a^2N_J(\bar{Z})+a(b,\bar{Z}^{\sharp})+\displaystyle\frac{1}{2}(b,b\times \bar{Z}))).
\]
\end{enumerate}
\end{thm}
\begin{proof}
We show the case of $\xi>0$ since the case of $\xi<0$ is settled in a quite similar manner. 
From (\ref{Diff-eq_gen1st-1}) and (\ref{Diff-eq_gen2nd-1}) we have $H_n=Cw^{-n}\exp(-2\pi\xi w^2)$ and $G_{-n}=0$ as a solution of moderate growth with respect to $w$.

We next consider (\ref{Diff-eq_gen2nd-2}), (\ref{Diff-eq_gen3rd-2}) and (\ref{Diff-eq_gen4th-1}). 
From these and Lemma \ref{Diffeq-rk1-Lem}
we obtain
\[
F_n=CN_J(Y)^{n/2}\exp(-2\pi\xi \sqrt{-1}(a^2N_J(Z)+a(b,Z^{\sharp})+\displaystyle\frac{1}{2}(b,b\times Z))).
\]
The equation (\ref{Diff-eq_gen4th-1}) implies $D_{Z^*(E)}\cdot(N_J(Y)^{-\frac{n}{2}}G_n)=0$. The above solution is compatible with this. As a result we have the explicit formula for $\phi_v$s in the assertion. 

\end{proof}
This result leads to important consequences on generalized Whittaker models and then yields a result on Fourier-Jacobi models~(or Fourier-Jacobi type spherical functions in the sense of Hirano \cite{Hi-1},~\cite{Hi-2}and \cite{Hi-3}). Given a generalized Whittaker function $W$ for $\pi_n$ with the minimal $K$-type $\tau_n$ let $\pi(W)$ be the $(\g,K)$-module generated by the coefficient functions of $\{(W(g),v)_{\tau_n}\mid v\in \bV_n\}$ with an inner product $(*,*)_{\tau_n}$ of $\tau_n$. When $\eta$ is a character, such coefficient functions are viewed as elements in
\[
C^{\infty}_{\eta}(N\backslash G):=\{F:\text{smooth function on $G$}\mid F(ng)=\eta(n)F(g)~\forall (n,g)\in N\times G\}.
\]
When $\eta=\eta_{\xi}$ or $\eta^{\infty}_{\xi}$ is a Schr\"odinger representation or its $C^{\infty}$-type representation with the central character parametrized by $\xi$, they are viewed as elements in
\[
C^{\infty}_{\psi_{\xi}}(N_m\backslash G):=\{F:\text{smooth function on $G$}\mid F(ng)=\psi_{\xi}(n)F(g)~\forall (n,g)\in N_m\times G\}
\]
where note that elements in the sections of $C^{\infty}\text{-}{\rm Ind}_N^G(\eta_{\xi})$ or $C^{\infty}\text{-}{\rm Ind}_N^G(\eta^{\infty}_{\xi})$ are regarded as those in $C^{\infty}\text{-}{\rm Ind}_{N_m}^G(\psi_{\xi})$ in view of the induction by stages. Toward the coming argument we need the following proposition:
\begin{prop}\label{WhittakerModule}
For any irreducible representations $\eta$ or the $C^{\infty}$ representations $\eta^{\infty}_{\xi}$ of $N$ just explained, let $W$ be as above and non-zero. 
We have an isomorphism $\pi(W)\simeq\pi_n$ as $(\g,K)$-modules.
\end{prop}
\begin{proof}
Let $\g=\lk\oplus\p$ be the Cartan decomposition of the Lie algebra $\g$ of $G$~(see \cite[Sections 5 and 6]{P1}), and $\lk_{\C}$ and ${\p_{\C}}$ denote the complexification of $\lk$ and $\p$ respectively. By $U(\g_{\C})$ and $U(\lk_{\C})$ we denote the universal enveloping algebra of $\g_{\C}$ and $\lk_{\C}$ respectively. Following \cite[Section 5.1]{P1} we introduce a basis $\{h_l,e_l,f_l\}$ of ${\mathfrak sl}_2(\C)$-factor of $\lk_{\C}$~(which forms an ${\mathfrak sl}_2$-triple) and generators $\{h_3,h_1(X),h_{-1}(Y),h_{-3},\overline{h_3},\overline{h_1(X')},\overline{h_{-1}(Y')},\overline{h_{-3}}\mid X,X',Y,Y'\in J\}$ of $\p_{\C}$ as a $\C$-vector space. We later need the subspace $\p^+$ of $\p_{\C}$ spanned by $\{h_3,h_1(X),h_{-1}(Y),h_{-3} \mid X,Y\in J\}$. This is an $M$-module and thus so is the homogeneous degree $n$-part $S^n(\p^+)$ of the symmetric tensor algebra of $\p^+$ for each positive integer $n$, where see Section \ref{GWF-definition} for the subgroup M of $K$.

Now recall that $\bV_n$ denotes the representation space of the minimal $K$-type of the quaternionic discrete series representation $\pi_{n}$.
We are going to construct a $(\g,K)$-module $S_{\bV_n}$ isomorphic to the quaternionic discrete series $\pi_{n}$. 
Toward this we need to note that there is the direct sum decomposition $\p_{\C}\otimes {\bV_n}=V^+\oplus V^-$ with $V^{\pm}\simeq\Sym^{2n\pm1}(\C^2)\oplus\p^+$ as $K$-modules, where $\Sym^r(\C^2)$ denotes the $r$-th symmetric tensor representation of the 2-dimensional standard representation of $SU(2)$ for a positive integer $r$. For this decomposition see \cite[Sections 7.1,~7.3]{P1}. 
Let us introduce the $(\g,K)$-module $U(\g_{\C})\otimes_{U(\lk_{\C})}{\bV_n}$ and the submodule $U(\g_{\C})\otimes_{U(\lk_{\C})}V^-$ of this. 
%
We then define 
\[
S_{\bV_n}:=U(\g_{\C})\otimes_{U(\lk_{\C})} {\bV_n}/U(\g_{\C})\otimes_{U(\lk_{\C})}V^-,
\]
which will be verified to be isomorphic to the quaternionic discrete series $\pi_n$ as a $(\g,K)$-module. 
 
To understand the structure of $S_{\bV_n}$ in detail, it is important to know $V^+\simeq \p_{\C}\otimes {\bV_n}/V^-$ explicitly. As a basis of ${\bV_n}$ we take $\{[x^{n+k}][y^{n-k}]:=\left(\frac{x^{n+k}}{(n+k)!}\right)\left(\frac{y^{n-k}}{(n-k)!}\right)\}_{-n\le k\le n}$ introduced in \cite[Section 7.3]{P1} and Section \ref{GWF-STrep}. Based on the contraction formula and the correspondence of vectors given in \cite[Section 7.3,~p1251]{P1} we then have the following formula:
\begin{align}
\overline{h_3}\otimes[x^{n+k}][y^{n-k}]&\equiv -h_{-3}\otimes[x^{n+k-1}][y^{n-k+1}]\mod V^-,\label{Contraction-1st}\\
\overline{h_1(X)}\otimes[x^{n+k}][y^{n-k}]&\equiv h_{-1}(X)\otimes[x^{n+k-1}][y^{n-k+1}]\mod V^-,\label{Contraction-2nd}\\
\overline{h_{-1}(Y)}\otimes[x^{n+k}][y^{n-k}]&\equiv -h_{1}(Y)\otimes[x^{n+k-1}][y^{n-k+1}]\mod V^-,\label{Contraction-3rd}\\
\overline{h_{-3}}\otimes[x^{n+k}][y^{n-k}]&\equiv h_{3}\otimes[x^{n+k-1}][y^{n-k+1}]\mod V^-\label{Contraction-4th},
\end{align}
where $X,Y\in J$.

We remark that, as a $K$-module, $S_{\bV_n}$ is generated by 
\[
\{h_3^{\otimes i}h_1(X)^{\otimes j}h_{-1}(Y)^{\otimes k}h_{-3}^{\otimes l}\otimes v',~v\mid (i,j,k,l)\in\Z_{\ge 0}^4,~X,Y\in J,~,v'\in V^+,~v\in {\bV_n}\},
\]
in view of (\ref{Contraction-1st})$\sim$(\ref{Contraction-4th}), the relation $[\p,\p]\subset\lk$ and the Poincar{\'e}-Birkhoff-Witt theorem of $U(\g_{\C})$~(cf.~\cite[Theorem 3.2]{Kn}), for which note that the relations (\ref{Contraction-1st})$\sim$(\ref{Contraction-4th}) generate the equivalence relation $\mod V^-$. 
We see that there is an isomorphism
\[
S_{\bV_n}\simeq \oplus_{m\ge 0}\Sym^{m+2n}(\C^2)\boxtimes S^m(\p^{+})
\]
as $K$-modules.
In fact, $V^+\simeq \p_{\C}\otimes {\bV_n}/V^-$~(respectively~${\bV_n}$) is viewed as a subspace of $S_{\bV_n}$ isomorphic to $\p^+\otimes\Sym^{2n+1}(\C^2)$~(respectively~the tensor product of $\Sym^{2n}(\C^2)$ and the trivial representation of $M$) as a $K$-module, and the $K$-module generated by 
$\{h_3^{\otimes i}h_1(X)^{\otimes j}h_{-1}(Y)^{\otimes k}h_{-3}^{\otimes l}\otimes v' \mid (i,j,k,l)\in\Z_{\ge 0}^4,~X,Y\in J,~,v'\in V^+\}$ with the fixed homogeneous degree $m=i+j+k+l$ is isomorphic to $S^{m}(\p^+)\boxtimes\Sym^{m+2n}(\C^2)$. We explain this as follows; since the appearance of the $S^m(\p^{+})$-factor is obvious, this is verified by thinking of the $\Sym^*(\C^2)$-factor of the possible highest weights with the help of \cite[Theorem 5.2.1 and Section 7.3, p1251]{P1}~(for a more detailed account of the $M$-module structure of $S^{m}(\p^+)$ , see \cite[Section 6]{G-W}).

On the other hand, we recall that the distribution of $K$-types of quaternionic discrete series is given explicitly by Gross-Wallach \cite[Proposition 5.7,~Section 6]{G-W}. According to this we see that 
\[
\pi_{n}|_K\simeq\oplus_{m\ge 0}\Sym^{m+2n}(\C^2)\boxtimes S^m(\p^{+}),
\]
where note that all the multiplicities of the representation $\Sym^{m+2n}(\C^2)\boxtimes S^m(\p^{+})$ is one.
As a result we have an isomorphism $\pi_{n}|_K\simeq S_{\bV_n}$ as $K$-modules. 

Suppose that we can show the irreducibility of $S_{\bV_n}$. We then see that $S_{\bV_n}\simeq\pi_n$ as $(\g,K)$-modules in view of the characterization of discrete series representations by the distribution of $K$-types (cf.~Schmid \cite[Theorem 1.3]{S2}). 
For this we remark that the characterization just mentioned requires the irreducibility of the $(\g,K)$-modules. The $(\g,K)$-module $\pi(W)$ generated by the coefficient functions of the non-zero generalized Whittaker function $W$ is then verified to be isomorphic to the quaternionic discrete series representation $\pi_{n}$. In fact, ${\bV_n}$ is isomorphic to the $K$-module generated by the coefficient functions of $W$. We can naturally extend this isomorphism to a non-zero surjective intertwining operator from $S_{\bV_n}$ to $\pi(W)$, and thus we know that $\pi(W)$ is non-zero and irreducible. For this we note that $W$ is annihilated by the Dirac-Schmid operator. We see that \cite[Section 7.3]{P1} is useful and can apply (\ref{Contraction-1st})$\sim$(\ref{Contraction-4th}) to $\pi(W)$.  We then verify $\pi(W)$ is isomorphic to the quaternionic discrete series as $(\g,K)$-modules. 

The proof is thus reduced to the following lemma:
\begin{lem}
The $(\g,K)$-module $S_{\bV_n}$ is irreducible.
\end{lem}
\begin{pf}
Suppose that $S_{\bV_n}$ has a non-zero proper sub-$(\g,K)$-module $U$. If $U$ includes ${\bV_n}$, the $U(\g)$-stability of $U$ implies that $S_{\bV_n}=U$. We can thus assume that there is the minimal non-negative integer $u_0$ such that $U\cap(S^{u_0}(\p^+)\otimes V^+)\not=\emptyset$, where $S^{u_0}(\p^+)\otimes V^+$ is regarded as a subspace of $S_{\bV_n}$. To deny such possibility we note the following bracket relations~(cf.~\cite[Sections 6.2,~7.3]{P1})
\[
[h_3,h_{-3}]=-2e_l,~[h_1(X),h_{-1}(Y)]=-2(X,Y)e_l,~[e_l,h]=0~\forall h\in\p^+,
\]
where $(X,Y)$ denotes the value of the trace pairing for $X,Y\in J$. We remark that the last relation follows from the fact that each non-zero element in $\p^+$ is a highest weight vector with respect to the ${\mathfrak sl}_2(\C)$-action~(cf.~\cite[Section 7.3]{P1}).  
Take a non-zero $v\in U\cap (S^{u_0}(\p^+)\otimes V^+)$. We claim that, with a suitable $Z\in\{h_3,h_1(X),h_{-1}(Y),h_{-3}\mid X,Y\in J\}$, we can find a non-zero vector $Z\cdot v\in U$ so that it splits into the non-zero $S^{u_0+1}(\p^+)\otimes V^+$-part and the non-zero $S^{u_0-1}(\p^+)\otimes V^+$-part in view of the bracket relations above, where we understand $S^{u_0-1}(\p^+)\otimes V^+={\bV_n}$ when $u_0=0$. 
Indeed, to see this, decompose $v$ into a linearly independent sum of vectors of the form $h_3^{\otimes i}h_1(X)^{\otimes j}h_{-1}(Y)^{\otimes k}h_{-3}^{\otimes l}\otimes w$ with $w\in V^+$ and distinct $(i,j,k,l)$s in $(\Z_{\ge 0})^4$~($w\in {\bV_n}$ is possible when $i=j=k=l=0$). The problem then turns out to be reduced to the case where $v$ is of the form $h_3^{\otimes i_0}h_1(X)^{\otimes j_0}h_{-1}(Y)^{\otimes k_0}h_{-3}^{\otimes l_0}\otimes w$ with a single $(i_0,j_0,k_0,l_0)\in\Z_{\ge 0}^4$. By considering the action of $h_3,~h_1(X),~h_{-1}(Y)$ or $h_{-3}$ we then verify the claim in terms of the appearance of $e_l$ in the first and second bracket relations above and of the third bracket relation for $e_l$, where $[\p^+,\p^+]\subset\lk$ is also used. For this we furthermore note that the action of $e_l$ on $v$ is reduced to the infinitesimal action of $e_l$ for the $V^+$-part (or ${\bV_n}$-part) of $v$ by noting the third bracket relation.
As a result, there is non-zero appearance of the $S^{u_0-1}(\p^+)\otimes V^+$-part  but this contradicts to the minimality of $u_0$. 
\end{pf}
We verify the lemma and have consequently proved the proposition.  
\end{proof}
We are now able to discuss the exact multiplicity formula for the various generalized Whittaker models~(or functions)~associated with irreducible unitary representations or weakly cyclic representations of $N$ other than trivial characters. 

Following \cite[Section 1.2,~p.1214]{P1} we introduce the cubic polynomial $p_{\chi}$ associated with a non-trivial character $\chi$ of $N$, defined on the Hermitian symmetric space $D_J\subset J\otimes_{\R}\C$ associated with $J$, as follows:
\begin{equation}\label{cubic-polynomial}
p_{\chi}(Z):=-\langle (a,b,c,d),(1,-Z,Z^{\sharp},-N_J(Z))\rangle=aN_J(Z)+(b,Z^{\sharp})+(c,Z)+d,
\end{equation}
where $\chi$ is parametrized by $(a,b,c,d)\in\cW_J(\Q)$. 
We remind readers that, for an irreducible unitary representation $\eta$ or a weakly cyclic representation (such as $\eta=\eta^{\infty}_{\xi}$) of $N$, an element of the image of 
\[
\Hom_{(\g,K)}(\pi_n,C^{\infty}\operatorname{-Ind}_N^G(\eta))\ni\Phi\mapsto \Phi\circ\iota\in\Hom_{K}(\tau_n,C^{\infty}\operatorname{-Ind}_N^G(\eta))
\]
is called a generalized Whittaker function on $G$~(see Section \ref{Cubicforms-Groups} for $G$) for $\pi_n$ with $K$-type $\tau_n$~(cf.~Section \ref{GWF-definition}). 
When $\eta$ is a character, say $\chi$, an explicit formula for such generalized Whitaker functions of moderate growth is given uniformly for $G$ above by Pollack \cite{P1} in terms of the $K$-Bessel function and the cubic polynomial $p_{\chi}$ on $D_J$. 
 It is obtained by solving the differential equation arising from the vanishing by the infinitesimal action of the Dirac-Schmid operator ${\cal D}$, and Pollack \cite{P1} shows that the solution of moderate growth is unique, up to constant multiple for non-trivial characters of $N$. 
 
\begin{thm}\label{Multiplicity-formula}
Let $G$ be as above and $\pi_n$ be a quaternionic discrete series representation. 
\begin{enumerate}
\item For a non-trivial character $\chi$ we have 
\begin{align*}
&\dim\Hom_{(\g,K)}(\pi_n,C^{\infty}\operatorname{-Ind}_N^G(\chi))=
\begin{cases}
2&(\text{$p_{\chi}$ has no zero point in $D_J$}),\\
0&(\text{$p_{\chi}$:otheriwse}),
\end{cases}\\
&\dim\Hom_{(\g,K)}(\pi_{n},C^{\infty}_{\rm mod}\operatorname{-Ind}_N^G(\chi))=
\begin{cases}
1&(\text{$p_{\chi}$ has no zero point in $D_J$}),\\
0&(\text{$p_{\chi}$:otheriwse}),
\end{cases}
\end{align*}
where $C^{\infty}_{\rm mod}\operatorname{-Ind}_N^G(\chi)$ denotes the space of  moderate growth sections in $C^{\infty}\operatorname{-Ind}_N^G(\chi)$. 

When $\dim\Hom_{(\g,K)}(\pi_{n},C^{\infty}\operatorname{-Ind}_N^G(\chi))\not=0$, the generalized Whittaker functions are given by linear combinations of $W_{\chi}(g):=\underset{-n\le v\le n}{\sum}W_v^{\chi}(g)\frac{x^{n+v}y^{n-v}}{(n+v)!(n-v)!}$ and $W'_{\chi}(g):=\underset{-n\le v\le n}{\sum}{W'}_v^{\chi}(g)\frac{x^{n+v}y^{n-v}}{(n+v)!(n-v)!}$ for $g\in H_J$ with
\begin{align*}
&W_v^{\chi}(g)=\left(\frac{|j(g,\sqrt{-1}1_J)p_{\chi}(g\cdot\sqrt{-1}1_J)|}{j(g,\sqrt{-1}1_J)p_{\chi}(g\cdot\sqrt{-1}1_J)}\right)^v\nu(g)^n|\nu(g)|K_v(|j(g,\sqrt{-1}1_J)p_{\chi}(g\cdot\sqrt{-1}1_J)|),\\
&{W'}_v^{\chi}(g)=\left(-\frac{|j(g,\sqrt{-1}1_J)p_{\chi}(g\cdot\sqrt{-1}1_J)|}{j(g,\sqrt{-1}1_J)p_{\chi}(g\cdot\sqrt{-1}1_J)}\right)^v\nu(g)^n|\nu(g)|I_v(|j(g,\sqrt{-1}1_J)p_{\chi}(g\cdot\sqrt{-1}1_J)|),
\end{align*}
where $K_v$~(respectively~$I_v$) denotes the $K$-Bessel function~(respectively~$I$-Bessel function) parametrized by $v$.  Here $j(g,Z)$ and $\nu(g)$ with $(g,Z)\in H_J\times\cH_J$ denote the factor of automorphy~(cf.~\cite[Proposition 2.3.1]{P1}) and the similitude factor as in \cite[Section 2.2]{P1}. 

When $\dim\Hom_{(\g,K)}(\pi_{n},C^{\infty}_{\rm mod}\operatorname{-Ind}_N^G(\chi))=1$, the generalized Whittaker function is explicitly given by $W_{\chi}(g)$ up to scalars.
\item For any $\xi\in\R\setminus\{0\}$ recall that $\eta^{\infty}_{\xi}$ and $\eta_{\xi}$ denote the $C^{\infty}$ Schr\"odinger representation and the usual Schr\"odinger representation respectively, with the central character parametrized by $\xi$~(cf.~Section \ref{Rep-JacobiGp}). Let $W_{\eta_\xi^{\infty},\tau_n}^{\rm mod-sep}$~(respectively~$W_{\eta_\xi,\tau_n}^{\rm mod-sep}$) be the $\C$-span of generalized Whittaker functions for $\pi_n$ with the minimal $K$-type $\tau_n$ attached to $\eta_{\xi}^{\infty}$~(respectively~$\eta_{\xi}$) which satisfy the assumption of the separation of variables~(just before Proposition \ref{Diff-eq_general-2}) and are of moderate growth with respect to $w$. We have
\[
\begin{cases}
\dim\Hom_{(\g,K)}(\pi_n,C^{\infty}_{\rm mod-sep}\text{-}{\rm Ind}_N^G(\eta^{\infty}_{\xi}))=\dim W_{\eta_\xi^{\infty},\tau_n}^{\rm mod-sep}= 1,\\
\dim\Hom_{(\g,K)}(\pi_n,C^{\infty}_{\rm mod-sep}\text{-}{\rm Ind}_N^G(\eta_{\xi}))=\dim W_{\eta_{\xi},\tau_n}^{\rm mod-sep}=0,
\end{cases}
\]
where, for $\eta=\eta_{\xi}$ and $\eta_{\xi}^{\infty}$, $C^{\infty}_{\rm mod-sep}\operatorname{-Ind}_N^G(\eta)$ denotes the submodule of $C^{\infty}\operatorname{-Ind}_N^G(\eta)$ generated by moderate growth sections satisfying the separation of variables. 
\end{enumerate}
\end{thm}
\begin{proof}
We begin with remarks on the explicit formulas for generalized Whittaker functions taken up in the statement of the theorem. 
Regarding the first assertion the moderate growth Whittaker function is obtained by Pollack \cite[Sections 8,~9 and 10]{P1}. The other Whitaker function, which is not of moderate growth, is obtained in quite a similar manner but note that the recurrence relations \cite[Lemma 8.1.1~(4),~(5)]{P1}~(equivalent to \cite[Lemma 8.1.1~(3),~(2)]{P1} respectively) differ by the minus for the case of the $I$-Bessel function. 

As for the second assertion, we should note that the non-zero generalized Whittaker functions for Schr\"odinger representations for non-zero $\xi$s~(cf.~Theorem  \ref{GenWhittakerFct-gen}) are $C^{\infty}(\R\oplus J)\boxtimes \bV_n$-valued, but not $L^2(\R\oplus J)\boxtimes \bV_n$-valued. In fact, we can check it, for instance, by noting that the quadratic term $-\sqrt{-1}(b,b\times Z)$ appearing in the explicit formula is positive definite on $\Omega$ when $Z=\sqrt{-1} y\cdot 1_J$ with $y>0$~(see the beginning of Section 2.3 in \cite{P1}), where $(b,b\times 1_J)=(b\times b,1_J)=2(b^{\sharp},1_J)$ by the sharp symmetry as in \cite[p191]{Mc})~($(b^{\sharp},1_J)$ is known as the spur form \cite[Section 3.8]{Mc}). 

Based on the explicit formulas for generalized Whittaker functions~(cf.~\cite{P1} and Theorem \ref{GenWhittakerFct-gen}) and the argument above, the multiplicity formulas in the theorem are verified by Proposition \ref{WhittakerModule} and \cite[Proposition 2.1]{Y}. 
\end{proof}
We are now able to established another important result, which respects the Fourier-Jacobi models~(or  Fourier-Jacobi type spherical functions in the sense of Hirano \cite{Hi-1},~\cite{Hi-2}and \cite{Hi-3})~for quaternionic discrete series $\pi_n$. 
\begin{thm}\label{FJ-model}
For any irreducible unitary representations $\rho_{\sigma,\xi}$ of the Jacobi group $R_J$,  
\[
\Hom_{(\g,K)}(\pi_n,C^{\infty}_{\rm mod}\text{-}{\rm Ind}_{R_J}^G\rho_{\sigma,\xi})=0
\]
holds, where $C^{\infty}_{\rm mod}\text{-}{\rm Ind}_{R_J}^G\rho_{\sigma,\xi}$ denotes the representation of $G$ with representation space given by moderate growth sections in $C^{\infty}\text{-}{\rm Ind}_{R_J}^G\rho_{\sigma,\xi}$.

\end{thm}
\begin{proof}
Let $\Phi\in \Hom_{(\g,K)}(\pi_n,C^{\infty}_{\rm mod}\text{-}{\rm Ind}_{R_J}^G\rho_{\sigma,\xi})$.  We begin with remarking that every $g\in G$ admits a decomposition $g=h_ga_gk_g$ with $h_g\in R_J,~a_g\in A_P,~k_g\in K$, where $A_P\simeq\{w\in\R_{>0}\}$ is the connected torus subgroup in the center of $H_J$. In addition to this, let $h_g=n_gs_g\in R_J$ be an $R_J$-factor of $g$ decomposed with $n_g\in N$ and $s_g\in H_J^{0,1}$. Take arbitrary non-zero vectors $u_0$ in the representation spaces of $\rho_{\sigma,\xi}$. 
Furthermore let $(*,*)_{\rho_{\sigma,\xi}}$ denote the inner product of $\rho_{\sigma,\xi}$. 
Let $(*,*)_{\sigma}$ be the canonical $V_{\xi}$-valued paring of $V_{\sigma}\boxtimes V_{\xi}\times V_{\sigma}$ induced by the inner product of $\sigma$.  
Choose $u_{00}\in V_{\sigma}$ such that $(u_0,u_{00})_{\sigma}\not=0$. 
Putting $V_{\pi_n}$ to be a representation space of $\pi_n$, we then have that
\[
\tilde{\Phi}_{u_0}:V_{\pi_n}\ni v\mapsto \tilde{\Phi}_{u_0}(v)(g):=(\Phi(\pi_n(g)v),\rho_{\sigma^*,\xi}(n_g)u_0)_{\rho_{\sigma,\xi}}(\rho_{\sigma,\xi}(n_g)u_0,u_{00})_{\sigma}\in C^{\infty}\text{-}{\rm Ind}_{N}^G\eta_{\xi}
\]
defines an element in $\Hom_{(\g,K)}(\pi_n,C^{\infty}\text{-}{\rm Ind}_N^G\eta_{\xi})$. This is a well-defined function in $g\in G$ since this is verified to be independent of the choice of the $R_J$-factor $h_g$ of $g$. In fact, we can show that the indeterminacy of the $h_g$-factor is given by $K\cap H_J^{0,1}$, which is absorbed in the $s_g$-factor, and that the $n_g$-factor of $h_g$ is thus uniquely determined for $h_g$.  

We now note that the Fourier-Jacobi type spherical functions are determined by the restriction to $A_P$ since $R_J\backslash G/K$ is represented by $A_P$. 
Suppose that $v=v_0$ is a vector of the minimal $K$-type. The argument of the previous paragraph implies that, if there exists $u_0$ such that $\{(\Phi(\pi_n(g)v_0),u_0)_{\rho_{\sigma,\xi}}\mid g\in G\}\not=0$ we get a non-zero element in $W_{\eta_\xi,\tau_n}^{\rm mod-sep}$, into which $\Hom_{(\g,K)}(\pi_n,C^{\infty}_{\rm mod-sep}\text{-}{\rm Ind}_N^G\eta_{\xi})$ can be embedded by Yamashita \cite[Proposition 2.1]{Y}.
Theorem \ref{GenWhittakerFct-gen} however implies $W_{\eta_\xi,\tau_n}^{\rm mod-sep}=\{0\}$. In addition, recalling the factorization $g=h_ga_gk_g$ with $h_g=n_gs_g$, we note that 
\begin{align*}
\tilde{\Phi}_{u_0}(v)(g)&=(\rho_{\sigma,\xi}(n_gs_g)\Phi(\pi_n(a_gk_g)v),\rho_{\sigma^*,\xi}(n_g)u_0)_{\rho_{\sigma,\xi}}(\rho_{\sigma,\xi}(n_g)u_0,u_{00})_{\sigma}\\
&=(\Phi(\pi_n(a_gk_g)v),\rho_{\sigma^*,\xi}(s_g^{-1})u_0)_{\rho_{\sigma,\xi}}(\eta_{\xi}(n_g)u_0,u_{00})_{\sigma}.
\end{align*}
We then see from the irreducibility of $\rho_{\sigma,\xi}$~(or of $\sigma$ and $\eta_{\xi}$) that the map $\Phi\mapsto\tilde{\Phi}_{u_0}$ is verified to be an injection into $\Hom_{(\g,K)}(\pi_n,C^{\infty}\text{-}{\rm Ind}_N^G\eta_{\xi})$ if $u_0$ satisfies the non-zero condition on $\{(\Phi(\pi_n(g)v),u_0)_{\rho_{\sigma,\xi}}\mid g\in G\}$ as above and that otherwise $\Phi\equiv 0$. Anyway this leads to $\dim\Hom_{(\g,K)}(\pi_n,C^{\infty}\text{-Ind}_{R_J}^G\rho_{\sigma,\xi})=0$. As a result we are done.
\end{proof}

\section{Fourier-Jacobi expansion~(in the adelic formulation)}\label{FJ-exp}
\subsection{A general theory}\label{FJ-exp-general}
Let $\cG$ be a simple $\Q$-algebraic group of adjoint type whose group of real points is $G$, where we retain the notation $G$ from the previous section. For this we assume that $J$ is a real vector space equipped with a $\Q$-structure given by a cubic norm structure $J_{\Q}$ over $\Q$, i.e. $J=J_{\Q}\otimes_{\Q}\R$. This yields the $\Q$-subgroup $\cN$ of $\cG$ defined by 
\[
\cN(\Q):=\{n((a,b,c,d),t)\mid (a,b,c,d)\in \cW_J(\Q),~t\in\Q\},
\]
with the affine space $\cW_J$ over $\Q$ defined by 
\[
\cW_J(\Q):=\Q\oplus J_{\Q}\oplus J^{\vee}_{\Q}\oplus\Q,
\] 
where $J^{\vee}_{\Q}$ denotes the dual of $J_{\Q}$ with respect to the trace paring~(cf.~Section \ref{Cubicforms-Groups}). This algebraic group $\cN$ is nothing but what should be called a Heisenberg group. In addition to this, we further introduce the $\Q$-subgroup $\cN_m$ of $\cN$ by
\[
\cN_m(\Q):=\{n((0,0,c,d),t)\mid (c,d)\in J^{\vee}_{\Q}\oplus\Q,~t\in\Q\}.
\]

We now recall that $\cG$ has the Heisenberg parabolic subgroup $\cP_J=\cN\rtimes\cH_J$~(cf.~\cite[Section 4.3.2]{P1}) whose unipotent radical is $\cN$~(respectively Levi part is $\cH_J$). By $\cR_J$ we denote the algebraic subgroup of $\cP_J$ given by the semi-direct product
\[
\cN\rtimes\cH_J^{1,0},
\]
with the connected component $\cH_J^{0,1}$ of the identity for the semi-simple part of the Levi subgroup of $\cP_J$. For example $\cH_J^{0,1}=SL(2)$ when $\cG$ is of type $G_2$. We naturally call $\cR_J$ the Jacobi group. 

Let us make some review on the Weyl group of $\cG$. Regarding the $G_2$-case, the root system of $G_2$ has the set $\{\alpha,\beta\}$ of simple roots, where $\alpha$~(respectively $\beta$) denotes the long root~(respectively~short root). We let $w_{\alpha}$ and $w_{\beta}$ be the simple reflections for $\alpha$ and $\beta$ respectively. 
For a general $\cG$ we refer to \cite[Section 2.2]{Ru}. For this we recall that there is the $\Z/3$-grading of the Lie algebra $\g_J$ defined over some general ground field~(cf.~\cite[Section 4.2]{P1}), from which we know without difficulty that the algebraic group of type $G_2$ can be embedded into a general $\cG$ and that there is the ${\mathfrak sl}_3$-factor in the $\Z/3$-grading. Then we see that  
\[
w_{12}=\begin{pmatrix}
0 & 1 & 0\\
-1 & 0 & 0\\
0 & 0 & 1
\end{pmatrix}\in SL_3
\]
corresponds to $w_{\alpha}$~(denoted by $s_{\alpha}$ in \cite[Section 2.2]{Ru}). Here we can view $w_{12}$ as an element of $\cG(\Q)$ since we verify that $w_{12}$ defines an adjoint action on the Lie algebra using its $\Z/3$-grading. In fact, $w_{12}$ defines an element of the Weyl group of $\cG$ given by the reflection of the simple root corresponding to the root vector $E_{12}\in {\mathfrak sl}_3$~(the Lie algebra of $SL_3$). There is also an element of Weyl group of $\cG$ corresponding to $w_{\beta}$ in $G_2$ explicitly given in \cite[p145]{Ru}~(denoted by $s_{\beta}$ there). 
In $\cG$ we also use the same notation $w_{\alpha}$ and $w_{\beta}$. In \cite[Section 2.2]{Ru} it is remarked that the Weyl group of $\cG$ is generated by $w_{\alpha},~w_{\beta}$ and the Weyl group for the automorphism group of $J$ preserving the cubic norm of $J$.

Let us now introduce the adele groups $\cG(\bA)$, $\cN(\bA)$, $\cN_m(\bA)$ and $\cR_J(\bA)$ etc. with the adele ring $\bA$ of $\Q$. We often denote every element $n((a,b,c,d),t)\in\cN(\bA)$ by $n(\nu,t)$ with $\nu=(a,b,c,d)\in \cW_J(\bA)$, where $\cW_J(\bA)$ denotes the adelization of $\cW_J(\Q)$. 

Let $F$ be a general automorphic form on the adele group $\cG(\bA)$. We do not assume the moderate growth condition on $F$ but keep other usual defining conditions such as the automorphy and real analyticity ensured by the $Z(\g)$-finiteness~(cf.~\cite[Chapter I, Section 2]{Ha},~\cite[Section I.2.17.]{Mo-Wa}), where $Z(\g)$ denotes the center of the universal enveloping algebra of $\g$. We discuss the Fourier-Jacobi expansion 
\[
F(g)=\sum_{\xi\in\Q}F_{\xi}(g),\quad F_{\xi}(g):=\int_{\Q\backslash\bA}F(n(0,t)g)\psi(-\xi t)dt\quad(g\in\cG(\bA))
\]
with $0\in \cW_J(\Q)$, 
where $\psi$ denotes a fixed non-trivial additive character of $\Q\backslash \bA$. It is known that if an irreducible unitary representation of $\cN(\bA)$ is trivial on the center it is a character of $\cN(\bA)$. So $F_0$ is verified to 
\[
F_0(g)=\sum_{\chi}F_{\chi}(g),~F_{\chi}(g):=\int_{\cN(\Q)\backslash\cN(\cA)}F(ng)\chi
(n)^{-1}dn,
\]
where $\chi$ runs over left $\cN(\Q)$-invariant characters of $\cN(\bA)$. 

We next prove that $F_{\xi}$ can be also expanded by characters of $\cN(\bA)$. More precisely, we need the element $w_{\alpha}$ of the Weyl group of $\cG$ in addition. We begin with the expansion of $F_{\xi}$ along $\cN_m(\Q)\backslash\cN_m(\bA)$ given as follows:
\[
F_{\xi}(g)=\sum_{\psi_{\xi}}F_{\psi_{\xi}}(g),\quad F_{\psi_{\xi}}(g):=\int_{\cN_m(\Q)\backslash\cN_m(\bA)}F(ng)\psi_{\xi}(n)^{-1}dn,
\]
where $\psi_{\xi}$ runs over characters of $\cN_m(\Q)\backslash\cN_m(\bA)$ such that $\psi_{\xi}(n(0,t))=\psi(\xi t)$. For this note that $\cN_m$ is abelian. We now need the following lemma:
\begin{lem}\label{adelic_Ch-Lem}
For $\xi\in\Q\setminus\{0\}$ let $\psi_{\xi,0}$ be the character of $\cN_m(\bA)$ defined by $\psi_{\xi,0}(n(w,,t))=\psi(\xi t)$ for $n(w,t)\in \cN_m(\bA)$.
\begin{enumerate}
\item Every character of $\cN_m(\Q)\backslash\cN_m(\bA)$ is of the form
\[
\cN_m(\bA)\ni n(w,t)\mapsto \psi_{\xi,0}(n(\nu,0)n(w,t)n(-\nu,0))=\psi(\xi(\langle\nu,w\rangle+t))
\]
with some $\nu\in\Q\oplus J_{\Q}$, where $\nu=(a,b)$ is canonically viewed as an element of $\cW_J(\Q)$ by identifying it with $(a,b,0,0)\in\cW_J(\Q)$.
\item We have $w_{\alpha}\cN_m(\Q)w_{\alpha}^{-1}=\cN_m(\Q)$.
\end{enumerate}
\end{lem}
\begin{proof}
The first assertion is verified by the direct calculation of the symplectic form $\langle*,*\rangle$ on $W_J$, which is closely related to the composition law of the Heisenberg group~(we can review it in Section \ref{Rep-JacobiGp}). The second assertion is verified by checking the $w_{\alpha}$-action on the Lie algebra for $\cN_m$. 
For a detailed description of the Lie algebra of $\cG$ see \cite[Sections 4.2.1 and 4.2.4]{P1} and \cite[Section 2.2]{Ru}, where \cite[Section 4.2.4 (4)]{P1} is useful for us. 
The adjoint action by $w_{\alpha}$ on the Lie algebra of $\cG$~(induced by $w_{12}$) gives rise to exchanging of the center of $\cN$ and the last entry of the vector part $\cW_J$ of $\cN$, which correspond  to $E_{13}$ and $E_{23}$ respectively in the ${\mathfrak sl}_3$-factor of the $\Z/3$-grading of the Lie algebra. The $w_{\alpha}$-action remains $\{n((0,0,c,0),0)\mid c\in J_{\Q}^{\vee}\}$ unchanged.
\end{proof}
For the coming argument we note again that the $w_{\alpha}$-action sends the center of $\cN$ to the last entry of the vector part $\cW_J$ of $\cN$ as is remarked in the proof of the lemma above . 
 From the above expansion of $F_{\xi}$ along $\cN_m(\Q)\backslash\cN_m(\bA)$ we have
\[
F_{\xi}(g)=\sum_{\mu\in\Q}\sum_{\psi_{\xi}^{(b)}}F_{\psi_{\xi}^{(b)}}(n((\mu,0,0,0),0)g),
\]
where $\psi_{\xi}^{(b)}$ ranges over characters of $\cN_m(\Q)\backslash\cN_m(\bA)$ given in Lemma \ref{adelic_Ch-Lem} part 1 with $\nu=(0,\beta,0,0)\in\cW_J(\Q)$. 
Here note that the parameter $\mu\in\Q$ is necessary in view of Lemma \ref{adelic_Ch-Lem} part 1. 
Let us put $\psi_{\xi,\alpha}^{(b)}$ by $\psi_{\xi,\alpha}^{(b)}(n):=\psi_{\xi}^{(b)}(w_{\alpha}nw_{\alpha}^{-1})$. We then have
\[
F_{\psi_{\xi}^{(b)}}(w_{\alpha}g)=F_{\psi_{\xi,\alpha}^{(b)}}(g),
\]
which is due to the part 2 of Lemma \ref{adelic_Ch-Lem}. This is proved to be trivial on the center of $\cN(\bA)$. We thus can expand $F_{\psi_{\xi,\alpha}^{(b)}}(g)$ by characters of $\cN(\Q)\backslash\cN(\bA)$. As a result we have obtained
\[
F_{\xi}(g)=\sum_{\mu\in\Q}\sum_{\chi_{\xi}}F_{\chi_{\xi}}(w_{\alpha}n((\mu,0,0,0),0)g)
\]
where $\chi_{\xi}$ runs over characters of $\cN(\Q)\backslash\cN(\bA)$ parametrized by $(\xi,\beta,\gamma,\delta)\in\cW_J(\Q)$. 
 We can therefore state the following:
\begin{thm}\label{adelic-FJ-exp}
Let $F_{\xi}$ be the $\xi$-term of the Fourier-Jacobi expansion of a general automorphic form $F$ for $\xi\in\Q\setminus\{0\}$. 
We have
\[
F_{\xi}(g)=\sum_{\mu\in\Q}\sum_{\chi'_{\xi}}F_{\chi_{\xi}}(w_{\alpha}n((\mu,0,0,0),0)g)=\sum_{\chi''_{\xi}}\Theta_{\chi''_{\xi}}(g)
\]
with
\[
\Theta_{\chi''_{\xi}}(g):=\sum_{\gamma\in\cN_m(\Q)\backslash\cN(\Q)}F_{\chi''_{\xi}}(w_{\alpha}\gamma g),
\]
where $\chi'_{\xi}$~(respectively~$\chi''_{\xi}$) runs over characters of $\cN(\Q)\backslash\cN(\bA)$ parametrized by $(\xi,\beta,\gamma,\delta)\in\cW_J(\Q)$~(respectively~$(\xi,0,\gamma',\delta')\in\cW_J(\Q)$). The function $\Theta_{\chi''_{\xi}}$ above converges due to the absolute convergence of the Fourier series and is left $\cR_J(\Q)$-invariant.
\end{thm}
\begin{proof}
The former equality on $F_{\xi}$ in the first assertion is already verified in the argument before the theorem. The latter equality is a matter of grouping for the summation in view of Lemma \ref{adelic_Ch-Lem} part 1. 
To verify the left $\cR_J(\Q)$-invariance of $\Theta_{\chi''_{\xi}}$ we remark that $\Theta_{\chi''_{\xi}}$ is the adelic  theta series with the test function $F_{\chi''_{\xi}}(w_{\alpha}*)$ attached to the Schr\"odinger representation $\eta_{\xi}$, e.g. due to Weil \cite{We2}. In fact, we have the bijection $\cN_m(\Q)\backslash\cN(\Q)\simeq \{(\alpha,\beta,0,0)\in\cW_J(\Q)\mid (\alpha,\beta)\in\Q\oplus J_{\Q}\}$, the latter of which is a Langrangian subspace of $\cW_J(\Q)$. 
To show the left $\cR_J(\Q)$-invariance of $\Theta_{\chi''_{\xi}}$ it suffices to see the left $\cH_J^{0,1}(\Q)$-invariance of $\Theta_{\chi''_{\xi}}$ since the left $\cN(\Q)$-invariance is obvious by the definition of $\Theta_{\chi''_{\xi}}$. 
For $\delta\in\cH_J^{0,1}(\Q)$ we note that $\Theta_{\chi''_{\xi}}(\delta g)$ is left $\cN(\Q)$-invariant. We then see that there is a constant $c(\delta)$ depending only on $\delta$~(independent of $\chi''_{\xi}$) such that 
\[
\Theta_{\chi_{\xi''_{\xi}}}(\delta g)=c(\delta)\Theta_{\chi''_{\xi}}(g),
\]
which is due to the uniqueness of a non-zero element in $\Hom_{\cN(\bA)}(\eta_{\xi},L^2(\cN(\Q)\backslash \cN(\bA))$ up to scalars. Together with the left $\cR_J(\Q)$-invariance of $F_{\xi}$ , this implies $c(\delta)=1$ for any $\delta\in\cH_J^{0,1}(\Q)$. As a result we have the left $\cR_J(\Q)$-invariance of $\Theta_{\chi''_{\xi}}(g)$. 
\end{proof}
As an immediate consequence of this theorem we have the following corollary, which is a generalization of \cite[Proposition 8.4]{G-G-S}. 
\begin{cor}\label{Determination-by-zeroth-coeff}
A general automorphic form $F$ on $\cG(\bA)$ is determined by $F_0$. That is, if two automorphic forms $F$ and $G$ on $\cG(\bA)$ satisfy $F_0\equiv G_0$, we then have $F\equiv G$.
\end{cor}
\subsection{Automorphic forms generating quaternionic discrete series and K\"ocher principle}\label{FJ-K-principle-contspec}
Let us think of the case of automorphic forms generating quaterinionic discrete series representations $\pi_n$ at the archimedean place. We can suppose that they generate the minimal $K$-type $\tau_n$ as $K$-modules, in other words, their weights are assumed to be given by $\tau_n$. 
Without assuming the moderate growth condition we now review the definition of quaternionic modular forms by Pollack~(cf.~\cite[Definition 1.1.1]{P1}) as follows:
\begin{defn}
A smooth $F:\cG(\Q)\backslash\cG(\bA)\rightarrow\bV_n^{*}$  is a modular form of weight $n$ if
\begin{enumerate}
\item $F(gk)=\tau_n^*(k)^{-1}F(g)$ for all $(g,k)\in\cG(\bA)\times K$,
\item $\cD\cdot F=0$,
\end{enumerate}
where recall that we can identify $(\tau_n,\bV_n)$ with its contragredient representation $(\tau_n^*,\bV_n^*)$~(cf.~Section \ref{GWF-definition}). 
\end{defn}
For such automorphic form $F$, consider the $G=\cG(\R)$ module 
\[
\pi(F):=\langle (F(*g),v)_{\tau_n}\mid g\in G,~v\in\bV_n\rangle
\]
with a $K$-invariant inner product $(*,*)_{\tau_n}$ of $(\tau_n,\bV_n)$. 
\begin{prop}\label{module_QMF}
As a $(\g,K)$-module $\pi(F)$ generates a quaternionic discrete series representation $\pi_n$
\end{prop}
\begin{proof}
This is verified by following the argument of the proof for Proposition \ref{WhittakerModule}.
\end{proof}
This implies that the second definition condition of the definition above can be replaced by ``$F$ generates a quaternionic discrete series representation $\pi_n$''. 
Conversely, if $F$ generates $\pi_n$, $\cD\cdot F=0$ is satisfied since the minimal $K$-type of $\pi_n$ is characterized in $\pi_n$ by the vanishing with respect to $\cD$ as the definition of $\cD$~(cf.~\cite{S1}) indicates.  

Let $||F(g)||_{\tau_n}$ denote the norm of $F(g)$ induced by the inner product $(*,*)_{\tau_n}$ for $g\in\cG(\bA)$. 
In what follows, $F_{00}$ denotes the Fourier transformation of $F$ by the trivial character of $\cN(\Q)\backslash\cN(\bA)$.

From detailed studies on the Fourier-Jacobi expansion for this case, we see that automorphic forms on $\cG(\bA)$ generating quaternionic discrete series representations have unique features. The most remarkable one is that the K\"ocher principle holds, i.e. they automatically satisfy the moderate growth condition. However, we should remark that we have to exclude the cases of $G_2$ and special orthogonal groups $SO(4,N)$~($PSO(4,N)$ for even $N$) due to the following proposition:
\begin{prop}\label{lower-bound-FJ}
Let $F$ be an automorphic form on $\cG(\bA)$ generating quaternionic discrete series and define $F_{00}$ by the invariant measure normalized so that the volume of $\cN(\Q)\backslash\cN(\bA)$ is one. We have
\[
||F(g)||_{\tau_n}\ge||F_{00}(g)||_{\tau_n}.
\]
This implies that one is not able to prove the K\"ocher principle for the exceptional group $G_2$ and the special orthogonal groups mentioned above.
\end{prop} 
\begin{proof}
From the Fourier-Jacobi expansion of $F$ 
we deduce 
the inequality 
\[
||F(g)||_{\tau_n}^2=\int_{\cN(\Q)\backslash\cN(\bA)}||F(g)||_{\tau_n}^2dn\ge ||\int_{\cN(\Q)\backslash\cN(\bA)}F_{00}(ng)dn||_{\tau_n}^2=||F_{00}(g)||_{\tau_n}^2
\]
with the normalized measure $dn$ of $\cN(\Q)\backslash\cN(\bA)$ as in the statement. As for the last comment in the statement, we note that, according to  \cite[Theorem 1.2.1 (2)]{P1}, $F_{00}$ is written in terms of elliptic modular forms or holomorphic modular forms on a direct product of the complex upper half plane and a symmetric domain of type $IV$, for which one is not able to prove the K\"ocher principle as is well known. Here we note that we do not impose the holomorphy around cusps on the defining condition on elliptic modular forms etc., which corresponds to the moderate growth condition when they are regarded as automorphic forms on Lie groups. 
\end{proof}
Toward the proof of the K\"ocher we remark that the moderate growth condition for $\cG(\bA)$ is reduced to that for the restriction to $G=\cG(\R)$ in view of the finiteness of class numbers of algebraic groups over number fields~(for this, see \cite[Section 4]{B-J}). Taking the Iwaswa decomposition of $G$~(cf.~\cite[Theorem 5.12]{Kn}) into account, we can verify the principle by further restricting to $H_J=\cH_J(\R)$.  Here recall that $\cH_J$ denotes the Levi part of the parabolic subgroup $\cP_J$~(cf.~Section \ref{FJ-exp}) and that we have denoted $\cH_J(\R)$ by $H_J$ in Sections \ref{Rep-JacobiGp},~\ref{GWF-definition}.

We have to make discussion without assuming the moderate growth property of $F$. In addition, we note that the adelic automorphy of $F$ is given by its left $\cG(\Q)$-invariance and right invariance with respect to an open compact subgroup $K_f$ of the group $\cG(\bA_f)$ of finite adeles as well as the right $K$ equivariance with respect to $\tau_n$. Let $\Gamma:=\cG(\Q)\cap K_f\cG(\R)$. Then the automorphy is reformulated as the left $\Gamma$-invariance and the right equivariance with respect to $\tau_n$ in the non-adelic setting. 
Let $\Gamma_J$ be $\cH_J^{0,1}(\Q)\cap\Gamma$. 
We remark that $\Gamma_J$ is an arithmetic subrgroup of $\cH_J^{0,1}(\R)$ since $\cH_J^{0,1}$ is a closed $\Q$-subgroup of $\cG$. 
\begin{lem}\label{Lemma-KoecherPl}
Let the notation be as above. For $\gamma\in\Gamma_J$ let $\gamma_{\infty}$ denote the archimedean factor of $\gamma$. We have
\[
F_{\chi\cdot\gamma}(g)=F_{\chi}(\gamma_{\infty}g)\quad(g\in H_J=\cH_J(\R)),
\]
where $\chi\cdot\gamma$ is defined by $\chi\cdot\gamma(n):=\chi(\gamma n\gamma^{-1})$.
\end{lem}
\begin{proof}
For $g\in H_J$ we have
\begin{align*}
F_{\chi\cdot\gamma}(g)&=\int_{\cN(\Q)\backslash\cN(\bA)}F(ng)\chi(\gamma n\gamma^{-1})^{-1}dn=\int_{\cN(\Q)\backslash\cN(\bA)}F(\gamma^{-1}n\gamma g)\chi(n)^{-1}dn\\
&=\int_{\cN(\Q)\backslash\cN(\bA)}F(n\gamma g)\chi(n)^{-1}dn=F_{\chi}(\gamma g)=F_{\chi}(\gamma_{\infty}g),
\end{align*}
where we use the right $K_f$-invariance of $F$ to verify the last equation. 
\end{proof}
Now we review the generalized Whittaker function in Theorem \ref{Multiplicity-formula} part 1, which is written by a linear combination of $W_{\chi}$ and $W'_{\chi}$.
We can write $F_{\chi}$ as
\[
F_{\chi}(g)=C_{\chi}W_{\chi}(g)+C'_{\chi}W'_{\chi}(g)\quad(g\in H_J),
\]
where $C_{\chi}$ and $C'_{\chi}$ denote constants dependent only on $\chi$. 
\begin{thm}\label{KoecherPl}
For this theorem we do not assume the moderate growth condition for automorphic forms.
\begin{enumerate}
\item  Let $F$ be an automorphic form on $\cG(\bA)$ generating a quaternionic discrete series representation at the archimedean place. Let $\xi\in\Q\setminus\{0\}$. We have 
\[
F_{\xi}(g)=\sum_{\nu\in\Q}\sum_{\chi_{\xi}}F_{\chi_{\xi}}(w_{\alpha}n((\mu,0,0,0),0)g),
\]
where $\chi_{\xi}$ ranges over characters of $\cN(\Q)\backslash\cN(\bA)$ parametrized by $(\xi,\beta,\gamma,\delta)\in\cW_J(\Q)$ such that the cubic polynomial $p_{\chi_{\xi}}$ has no zero point on $D_J$~(cf.~\cite[Corollary 1.2.3]{P1}), where see (\ref{cubic-polynomial}) for the notation of the cubic polynomial.
\item (K\"ocher principle)~Automorphic forms generating quaternionic discrete series automatically satisfy the moderate growth condition except for the cases of $G_2$ and the special orthogonal groups of signature $(4,N)$.
\end{enumerate}
\end{thm}
\begin{proof}
Once the second assertion is verified, the first assertion is an immediate consequence of \cite[Corollary 1.2.3]{P1} and Theorem \ref{adelic-FJ-exp}. The coming argument is devoted to the second assertion, namely the proof of the K\"ocher principle.
 
We can say that the problem will be reduced essentially to the automatic validity of the moderate growth condition for the constant term $F_{00}$. 
We explain how to prove the moderate growth property of $F_0-F_{00}$ and $\sum_{\xi\not=0}F_{\xi}$ in detail, which leads to the reduction of the problem just mentioned. 
\subsection*{Elimination of extra terms}
If $F_{\chi}\not\equiv 0$, Lemma \ref{Lemma-KoecherPl} implies that $F_{\chi\cdot\gamma}\not\equiv 0$ holds for any $\gamma\in\Gamma_J$ and 
that 
\[
\sum_{\gamma\in\Gamma_J}F_{\chi\cdot\gamma}(g)=C_{\chi}\sum_{\gamma\in\Gamma_J}W_{\chi}(\gamma_{\infty}g)+C'_{\chi}\sum_{\gamma\in\Gamma_J}W'_{\chi}(\gamma_{\infty}g)\quad(g\in H_J)
\]
contributes non-trivially to the Fourier expansion of $F_0$. 
If $\chi$ is a non-trivial character, the cardinality of the $\Gamma_J$-orbit of $\chi$ is infinite. In fact, if $\chi$ is non-trivial 
there is $\delta\in\cH_J^{0,1}(\Q)$ such that $\chi\cdot\delta$ is parametrized by  $(1,0,c,d)$ with some $(c,d)\in J_{\Q}\oplus\Q$, 
This fact is due to \cite[Lemma 4.3.5]{P0}. 
The infinitude of a $\Gamma_J$-orbit of a non-trivial $\chi$ is verified by considering the action of $\{n(x),~n^{\vee}(y)\mid n(x),~n^{\vee}(y)\in\Gamma_J\cap\delta^{-1}\Gamma_J\delta\}$, where see \cite[Section 2.2]{P1} for the notation $n(x),~n^{\vee}(y)$. For this note that $\Gamma_J\cap\delta^{-1}\Gamma_J\delta$ is an arithmetic subgroup. 

For $\chi$ such that $p_{\chi}$ has a zero point in $D_J$
, Pollack \cite[Proposition 8.2.4]
{P1} showed that the generalized Whittaker functions should be zero. In fact, such a Whittaker function has singularity arising from a zero point of $p_{\chi}$ if we assume that it is non-zero. We should then remark that the argument of \cite[Proposition 8.2.4]{P1} just explained works without assuming the moderate growth condition. 
The generalized Whittaker functions for such $\chi$ do not therefore  appear in the Fourier expansion of $F_0$. 

We further note that $F_{\chi}$ satisfies $C'_{\chi}=0$ since otherwise it contradicts to the convergence of the Fourier series of $F_0$, where note that $W'_{\chi}$ is obviously unbounded since the $I$-Bessel function grows exponentially as is well known. 
More precisely, the most definite way to see this is to focus on the coefficient function for $v=0$ since there arises no signature change due to the factor $\left(
\pm\frac{|j(g,Z)p_{\chi}(Z)|}{j(g,Z)p_{\chi}(Z)}\right)^v$. For this we also remark that the $K$-Bessel functions and $I$-Bessel functions with non-negative integer parameters $v$ are positive on $\R_{>0}$. 

As a result, $W_{\chi}$ contributes non-trivially to the Fourier expansion of $F_0$ only if $\chi$ is trivial or parametrized by an element of $\cW_J(\Q)$ satisfying $p_{\chi}(Z)\not=0$ for any $Z\in D_J$. 
In \cite[Corollary 1.2.3,~Proposition 11.1.1]{P1}, when $\chi$ is trivial, $W_{\chi}=F_{00}$ is proved to be a sum of a constant function, a holomorphic modular form and its translate by some element $w_0$ of order two. See Proposition \ref{lower-bound-FJ} 
for the remark on the K\"ocher principle for the cases of $G_2$ and special orthogonal groups of signature $(4,N)$. 

\subsection*{Outline of the rest of the proof and preparation toward it}
Now we provide a rough explanation for the rest of the proof. 
We will firstly prove the moderate growth property of $F_0-F_{00}$ with the constant term $F_{00}$. This is the sum of $W_{\chi}$s over $\chi$ 
with the condition on $p_{\chi}$ above. It turns out to be at most of moderate growth, for which note that the summation ranges over a discrete subset included in a lattice of $\cW_J(\Q)$. We will next verify similar property of $F-F_0=\sum_{\xi\not=0}F_{\xi}$ by carrying out an estimation more or less  similar to that of $F_0-F_{00}$ by using the general theory of the Fourier-Jacobi expansion stated as Theorem \ref{adelic-FJ-exp}, whose validity does not need assuming the moderate growth condition. 
To verify the moderate growth property of $F_0-F_{00}$ we use the coordinate $g\in H_J$ while we choose the different coordinate $w_{\alpha}^{-1}gw_{\alpha}$ with $g\in H_J$ for $\sum_{\xi\not=0}F_{\xi}$. We remark that these different choice of the  coordinates do not influence the moderate growth property. 

Toward a more detailed argument we collect several facts that we use for the proof:
\begin{itemize}
\item From \cite[Lemma 2.3.2]{P1}, we introduce a positive definite quadratic form $(v,v)_J:=|\langle v,r_0(\sqrt{-1}1_J)\rangle|^2$ for $v=(a,b,c,d)\in W_J$, with the following notation~(cf.~\cite[Section 2.3]{P1})
\[
r_0(Z):=(1,-Z,Z^{\sharp},-N_J(Z))\quad(Z\in D_J).
\]
For $g=\nu(g)g_1$ with $g_1\in H_J^{0,1}$ and the similitude factor $\nu(g)\in\R^{\times}$ viewed also as an element in the center of $H_J$, we have
\begin{equation}\label{norm-equality}
|\langle v,g\cdot r_0(\sqrt{-1}1_J)\rangle|=|\nu(g)||\langle v,g_1\cdot r_0(\sqrt{-1}1_J)\rangle|=|\nu(g)||j(g_1,\sqrt{-1}1_J)\langle v,r_0(g_1\cdot\sqrt{-1}1_J)\rangle|
\end{equation}
(cf.~Proposition \cite[Proposition 2.3.1]{P1}). Here recall that we have introduced the factor $j(g,Z)$ of automorphy at the first assertion of Theorem \ref{Multiplicity-formula}.
\item For $v\in W_J$ we put $||v||:=\sqrt{(v,v)_J}$. For $g\in H_J$ we define an operator norm $||g||_{\infty}$ by 
\[
||g||_{\infty}:={\sup }_{v\in W_J}\frac{||gv||}{||v||}={\rm sup}_{||v||=1}||gv||.
\]
We then have
\begin{equation}\label{operatornorm-inequality}
\frac{||v||}{||g||_{\infty}}\le||g^{-1}v||.
\end{equation}
\item Let $x_{\alpha}(\mu):=n((\mu,0,0,0),0)$, which is the image of the exponential map of the root vector corresponding to $\alpha$. Then $w_{\alpha}x_{\alpha}(\mu)w_{\alpha}^{-1}=x_{-\alpha}(-\mu)$~(cf.~\cite[Proposition 6.4.3]{Ca}) and $\{h_{\alpha}(\nu),x_{\alpha}(\mu_1),x_{-\alpha}(\mu_2)\mid\nu\in\R_{>0},~\mu_1,~\mu_2\in\R\}$ generates $SL_2(\R)$ embedded into $G$~(see also \cite[Section 2]{JS}), where we use the notation $h_{\alpha}$ as in \cite[Section 2.3~(7)]{JS}. In this $SL_2(\R)$~(inside $G$), $x_{-\alpha}(-\mu)$ admits an Iwasawa decomposition 
\[
x_{\alpha}(-\frac{\mu}{\mu^2+1})h_{\alpha}(1/\sqrt{1+\mu^2})k_{\mu}=k_{\mu}h_{\alpha}(\sqrt{1+\mu^2})x_{\alpha}(-\frac{\mu}{1+\mu^2})=k_{\mu}x_{\alpha}(-\mu)h_{\alpha}(\sqrt{1+\mu^2})
\]
with the $SO(2)$-part $k_{\mu}$ corresponding to 
$\begin{pmatrix}
1/\sqrt{1+\mu^2} & \mu/\sqrt{1+\mu^2}\\
-\mu/\sqrt{1+\mu^2} & 1/\sqrt{1+\mu^2}
\end{pmatrix}$. 
\item According to \cite[Section 1.6~(6.6)]{Bu} we have
\begin{equation}\label{K-Bessel-inequality}
|K_v(y)|<\exp(-\frac{y}{2})K_{v}(2)\quad\text{if $y>4$}.
\end{equation}
\end{itemize}
In the subsequent argument we give some upper bounds of $F_0-F_{00}$ and $F-F_0$ by virtue of (\ref{K-Bessel-inequality}). We would be able to bound them without using (\ref{K-Bessel-inequality}). However, the bounds become more accessible if they are expressed in terms of the exponential function. 
\subsection*{A uniform bounding of $C_{\chi}$}
To obtain the estimate of $F-F_{00}$, we need that of $C_{\chi}$ for a non-trivial character $\chi$ contributing to $F_0$. 
We first note that 
\[
\int_{N\cap\Gamma\backslash N}(F(ng_0),w_v)_{\tau_n}\chi(n)^{-1}dn=C_{\chi}W_v^{\chi}(g_0)
\]
with an arbitrary fixed $g_0\in H_J$, where $w_v\in\bV_n$ is dual to $\frac{x^{n+v}y^{n-v}}{(n+v)!(n-v)!}$. The appropriate choice of $g_0$ is such that $F(g_0)\not=0$. 
We use the Schwartz inequality to obtain 
\[
|\displaystyle\int_{N\cap\Gamma\backslash N}(F(ng_0),w_v)_{\tau_n}\chi(n)^{-1}dn|\le\displaystyle\int_{N\cap\Gamma\backslash N}|(F(ng_0),w_v)_{\tau_n}|dn<C_{F,g_0}
\]
with a positive constant $C_{F,g_0}$ depending only on $F$ and $g_0$ since $F(ng_0)$ is bounded as a function on $N\cap\Gamma\backslash N$. From this and the explicit formula for $W_{\chi}$ in Theorem \ref{Multiplicity-formula} we deduce
\[
|C_{\chi}||\nu(g_0)|^{n+1}K_v(|j(g_0,\sqrt{-1}1_J)\langle v_{\chi}, r_0(g_0\sqrt{-1}1_J)\rangle|)<C_{F,g_0}.
\]
Because the set of non-trivial $\chi$s contributing to the Fouier expansion of $F$ is discrete and $K_v$ is a positive function bounded away from zero, there exists a character $\chi_0$ such that $K_v(|j(g_0,\sqrt{-1}1_J)\langle v_{\chi_0}, r_0(g_0\cdot \sqrt{-1}1_J)\rangle|)$ is maximal among such $\chi$s. We thus have
\begin{align*}
&|C_{\chi}||\mu(g_0)|^{n+1}K_v(|j(g_0,\sqrt{-1}1_J)\langle v_{\chi}, r_0(g_0\cdot\sqrt{-1}1_J)\rangle|)\\
\le& |C_{\chi}||\mu(g_0)|^{n+1}K_v(|j(g_0,\sqrt{-1}1_J)\langle v_{\chi_0}, r_0(g_0\cdot\sqrt{-1}1_J)\rangle|)<C_{F,g_0},
\end{align*}
which yields
\begin{equation}\label{FourierCoeff-Estimate}
|C_{\chi}|<\frac{C_{F,g_0}}{|\nu(g_0)|^{n+1}K_v(|j(g_0,\sqrt{-1}1_J)\langle v_{\chi_0}, r_0(g_0\cdot\sqrt{-1}1_J)\rangle|)}.
\end{equation}

\subsection*{Estimate of $F_0-F_{00}$}
For the rest of the argument we remind readers that we have let $||*||_{\tau_n}$ be the norm induced by the inner product for the unitary representation $(\tau_n,\bV_n)$. 
To come to the next step, for $g=\nu(g)g_1\in H_J$ with the notation $\nu(g)$ and $g_1$ above, we deduce
\begin{align*}
|\langle v_{\chi},gr_0(\sqrt{-1}1_J)\rangle|&=|\nu(g)||\langle v_{\chi},g_1r_0(\sqrt{-1}1_J)\rangle|=|\nu(g)|\cdot|\langle g_1^{-1}v_{\chi},r_0(\sqrt{-1}1_J)\rangle|\\
&=|\nu(g)|\cdot||g_1^{-1}v_{\chi}||\ge |\nu(g)|\frac{||v_{\chi}||}{||g_1||_{\infty}}.
\end{align*}
from (\ref{norm-equality}) and (\ref{operatornorm-inequality}). 
Now let us fix $g\in H_J$ arbitrarily. We first assume that 
\[
4<\frac{|\nu(g)|}{||g_1||_{\infty}} ||v_{\chi}||\quad(\le |\langle v_{\chi},gr_0(\sqrt{-1}1_J)\rangle|),
\]
where $v_{\chi}$ denotes the element of $\cW_J(\Q)$ corresponding to $\chi$. 
Under this assumption, bearing (\ref{K-Bessel-inequality}) and (\ref{FourierCoeff-Estimate}) in mind, we have
\begin{align*}
||F_0(g)-F_{00}(g)||_{\tau_n}\ll_{F,\tau_n,g_0} &|\nu(g)|^{n+1}\sum_{v_{\chi}}\exp(-\frac{1}{2}|\langle v_{\chi}, g\cdot r_0(\sqrt{-1}1_J)\rangle|)\\
\le &|\nu(g)|^{n+1}\sum_{v_{\chi}}\exp(-\frac{1}{2}\frac{|\nu(g)|}{||g||_{\infty}} ||v_{\chi}||),
\end{align*}
where $v_{\chi}$ or $\chi$ runs over the $N\cap\Gamma$-invariant characters satisfying the condition on $p_{\chi}$ as in the assertion. Here the implied constant of the first inequality depends only on $F,\tau_n$ and $g_0$. The summation above is estimated to be bounded, for which note that $\frac{|\nu(g)|}{||g||_{\infty}} ||v_{\chi}||>4$ by the assumption. 

Let us next assume that $\frac{|\nu(g)|}{||g_1||_{\infty}} ||v_{\chi}||\le 4$. The number of $\chi$ such that $\frac{|\nu(g)|}{||g_1||_{\infty}} ||v_{\chi}||\le 4$ is finite and is at most of polynomial order with respect to $g$.   
The summation of $F_{\chi}$ over such $v_{\chi}$ is at most of polynomial order with respect to $g\in H_J$ in view of the moderate growth property of $W_{\chi}$ as is proved in \cite[Proposition 8.2.4]{P1}. For this we note that $|C_{\chi}|$ has a uniform bound as in (\ref{FourierCoeff-Estimate}) and that the upper bounds of $||F_{\chi}(g)||_{\tau_n}$ thus depend only on the growth of the $K$-Bessel function $K_v$ on the range $(0,4]$. Indeed, the $K$-Bessel on that range increases around $0$ but is at most of polynomial order, hence we need only the upper bound of $||F_{\chi}(g)||_{\tau_n}$ for $\chi$ with minimal $||v_{\chi}||$~(which is non-zero) to estimate the summation under the current assumption. 

\subsection*{Estimate of $F-F_0$ and conclusion}
We now consider the estimate $F-F_{0}=\sum_{\xi\not=0}F_{\xi}$. We carry out this with the coordinate $w_{\alpha}^{-1}gw_{\alpha}$. 
As the first step we estimate $\sum_{\xi\not=0}F_{\xi}$ by
\begin{align*}
&||\sum_{\xi\not=0}F_{\xi}(w_{\alpha}^{-1}gw_{\alpha})||_{\tau_n}\le \sum_{\chi,\mu}||F_{\chi}(w_{\alpha}x_{\alpha}(\mu)w_{\alpha}^{-1}gw_{\alpha})||_{\tau_n}=\sum_{\chi,\mu}||F_{\chi}(k_{\mu}x_{\alpha}(-\mu)h_{\alpha}(\sqrt{1+\mu^2})gw_{\alpha})||_{\tau_n}\\
&\ll_F \sum_{\chi,\mu}||\chi(x_{\alpha}(-\mu))F_{\chi}(h_{\alpha}(\sqrt{1+\mu^2}) g))||_{\tau_n}\ll_F \sum_{\chi,\mu}||F_{\chi}(h_{\alpha}(\sqrt{1+\mu^2})g)||_{\tau_n},
\end{align*}
where $\chi$ or $v_{\chi}$ runs over elements of $\cW_J(\Q)$ with non-zero first entries satisfying the same condition as in the previous case and $\mu\in\Q$ over a suitable lattice in $\Q$ whose integral structure is induced by $\Gamma$. 
Here we make additional explanation on the above estimation. 
\begin{itemize}
\item The good analytic property of $F_{\chi}$~(real analytic property and moderate growth condition) implies that $||F_{\chi}(k_{\nu}x)||_{\tau_{n}}\ll_{F}||F_{\chi}(x)||_{\tau_{n}}$ holds for $x\in G$ with the implied constant depending only on $F$ since $k_{\mu}$ runs over a sequence inside the compact set $SO(2)$. Indeed, the upper bound of $||F_{\chi}(x)||_{\tau_n}$ is independent of $\chi$, as has been remarked above~(see ``A uniform bounding of $C_{\chi}$'' and ``Estimate of $F_{0}-F_{00}$) in terms of the growth of the $K$-Bessel function $K_v(y)$. The bound becomes better for large $y$~(e.g. $y>4$ as in (\ref{K-Bessel-inequality}). 
\item For the second inequality the $SO(2)$-part $k_{\mu}$ of the Iwasawa decomposition of $x_{-\alpha}(-\mu)$ can be therefore absorbed in the implied constant depending on $F$. We also note $w_{\alpha}\in K$. Then $||F_{\chi}(xw_{\alpha})||_{\tau_n}=||F_{\chi}(x)||_{\tau_n}$ follows from the right $K$-invariance of $||F_{\chi}(x)||_{\tau_n}$.
\end{itemize}

Now, given an arbitrary fixed $g\in H_J$, let us assume 
\[
4< 
\frac{(1+\mu^2)|\nu(g)|}{||g_1||_{\infty}} ||v_{\chi_0}|| 
\]
for characters $\chi$ and $\mu$s. 
 In the reasoning similar to the estimation of $F_0-F_{00}$ we continue our estimation as
\begin{align*} 
\sum_{\chi,\mu}||F_{\chi}(h_{\alpha}(\sqrt{1+\mu^2}g)||_{\tau_n}
&\ll_{F,\tau_n,g_0}|\nu(g)|^{n+1}\sum_{v_{\chi},\mu}(1+\mu^2)^{n+1}\exp(-\frac{1}{2}\frac{|\nu(h_{\alpha}(\sqrt{1+\mu^2})g)|}{||g_1||_{\infty}}\cdot ||v_{\chi}||)\\
&=|\nu(g)|^{n+1}\sum_{v_{\chi},\mu}(1+\mu^2)^{n+1}\exp(-\frac{1}{2}\frac{|\nu(g)|(1+\mu^2)}{||g_1||_{\infty}} ||v_{\chi}||),
\end{align*}
which is estimated to be bounded.

We are left with the case of $\frac{|\nu(g)|(1+\mu^2)}{||g_1||_{\infty}} ||v_{\chi_0}||\le 4$.  
The number of $\chi$ and $\mu$ with this condition is finite and of polynomial order with respect to $g$. As in the later part of the argument for the case of $F_0-F_{00}$ a simliar estimate therein leads to the polynomial growth property of the sum over such $\chi$ and $\mu$. 
 As a result, $F-F_{0}$ also thereby turns out to be at most polynomial order.

We have therefore verified that the moderate growth property of $F$ is reduced to that of $F_{00}$. 
Taking Proposition \ref{lower-bound-FJ} into account, we have consequently proved the assertion.
\end{proof}

Before our discussion for the case of cusp forms we remark that, according to \cite[Corollary 1.2.3]{P1}, the Fourier expansion of $F_0$ is contributed by characters parametrized by elements in $\cW_J(\Q)$ of rank four~(for the definition of the ranks see \cite[Definition 4.3.2]{P0}) satisfying the negativity with respect to Freudenthal's  quartic forms for cusp forms generating quaternionic discrete series, from which we also know $F_{\xi}$ of the cusp forms for $\xi\in\Q\setminus\{0\}$ in view of the theorem above or Theorem \ref{adelic-FJ-exp}. 

We are going to provide another aspect of the Fourier-Jacobi expansion for the case of cusp forms. 
To this end we introduce the notion of the discrete spectrum and the continuous spectrum for $L^2_{\xi}(\cR_J(\Q)\backslash\cR_J(\bA))$ with $\xi\in\Q\setminus\{0\}$. For $\xi\in\Q\setminus\{0\}$ we introduce the global Schr\"odinger representation $\eta_{\xi,\bA}$ of $\cN(\bA)$ by the $L^2$-induction
\[
\eta_{\xi,\bA}:=L^2\text{-Ind}_{\cN_m(\bA)}^{\cN(\bA)}\psi_{\xi,0},
\]
whose representation space is identified with $L^2(\bA\oplus J_{\bA})$~(for $\psi_{\xi,0}$ see Lemma \ref{adelic_Ch-Lem}), where $J_{\bA}$ denotes the adelization of $J_{\Q}$. It is well known that this representation extends to the Weil representation $\Omega_{\xi,\bA}$ of $Mp(W_J(\bA))\rtimes\cN(\bA)$ with the same representation space~(cf.~\cite{We1}). We additionally introduce 
\[
L^2_{\xi}(\cN(\Q)\backslash\cN(\bA)):=\{\phi\in L^2(\cN(\Q)\backslash\cN(\bA))\mid \phi(n(0,t)h)=\psi(\xi t)\phi(h)~\forall(t,h)\in\bA\times\cN(\bA)\},
\]
which is nothing but the image of $\Hom_{\cN(\bA)}(\eta_{\xi,\bA},L^2(\cN(\Q)\backslash \cN(\bA))$, which is a one-dimensional space spanned by $\theta_{\xi}$ defined as a theta series
\[
\theta_{\xi}(f)(n):=\sum_{\lambda\in \Q\oplus J_{\Q}}(\eta_{\xi,\bA}(n)f)(\lambda)
\]
for $f\in L^2(\bA\oplus J_{\bA})$~(cf.~\cite[Theorem 11]{Mr}). 

Let $\widetilde{\cH_J^{0,1}}(\bA)$ be the pull-back of $\cH_J^{0,1}(\bA)$ by the covering map $Mp(\cW_J(\bA))\rightarrow Sp(\cW_J(\bA))$. As $\eta_{\xi,\bA}$ can extends to $\Omega_{\xi,\bA}$, $L^2_{\xi}(\cN(\Q)\backslash\cN(\bA))$ become another realization of the global Weil representation by
\[
\theta_{\xi}(f)(ns):=\sum_{\lambda\in \Q\oplus J_{\Q}}(\Omega_{\xi,\bA}(ns)f)(\lambda)\quad((n,s)\in\cN(\bA)\rtimes Mp(W_J(\bA))
\]
for $f$ as above. As is well known~(cf.~\cite[Th\'eor\`eme 6]{We1}) this is left $\cN(\Q)\rtimes Sp(W_J(\Q)$. By $\omega_{\xi,\bA}$ we denote the representation of  
$\cN(\bA)\rtimes\widetilde{\cH_J^{0,1}}(\bA)$ obtained by the restriction of this global Weil representation to $\cN(\bA)\rtimes\widetilde{\cH_J^{0,1}}(\bA)$. 
We use the same notation $\omega_{\xi,\bA}$ to denote the representation of $\cN(\bA)\rtimes\widetilde{\cH_J^{0,1}}(\bA)$ on $L^2_{\xi}(\cN(\Q)\backslash\cN(\bA))$ induced by $\omega_{\xi,\bA}$.
\begin{prop}\label{decomp-Jacobifm}
Let $L^2(\cH_J^{0,1}(\Q)\backslash\widetilde{\cH_J^{0,1}}(\bA))_{-}$ be the space of square integrable automorphic forms with the same multiplier system as that of $\omega_{\xi,\bA}$. 
We have an isomorphism
\[
L^2_{\xi}(\cR_J(\Q)\backslash\cR_J(\bA))\simeq L^2(\cH_J^{0,1}(\Q)\backslash\widetilde{\cH_J^{0,1}}(\bA))_{-}\otimes L_{\xi}^2(\cN(\Q)\backslash\cN(\bA))
\]
as unitary representations of $\cR_J(\bA)$, where $L_{\xi}^2(\cN(\Q)\backslash\cN(\bA))$ is viewed as the representation space of $\omega_{\xi,\bA}$.
\end{prop}
\begin{proof}
We consider the linear injection 
\begin{align*}
&L^2(\cH_J^{0,1}(\Q)\backslash\widetilde{\cH_J^{0,1}}(\bA))_{-}\otimes L_{\xi}^2(\cN(\Q)\backslash\cN(\bA))\ni \phi\otimes h\mapsto\\
&\{\widetilde{\cH_J^{0,1}}(\bA)\times\cN(\bA)\ni(s,n)\mapsto\phi(s)\cdot\omega_{\xi,\bA}(s)h(n)\}\in L^2_{\xi}(\cR_J(\Q)\backslash\cR_J(\bA))
\end{align*}
for $(s,n)\in \widetilde{\cH_J^{0,1}}(\bA)\times\cN(\bA)$. By the definition of $L^2(\cH_J^{0,1}(\Q)\backslash\widetilde{\cH_J^{0,1}}(\bA))_{-}$, the evaluation by $s\in \widetilde{\cH_J^{0,1}}(\bA)$ is reduced to that for $\cH_J^{0,1}(\bA)$, namely $\phi(s)\cdot\omega_{\xi,\bA}(s)h(n)$ becomes a function on $\cR_J(\Q)\backslash\cR_J(\bA)$. 

To show that this is also a surjection let $\Phi\in L^2_{\xi}(\cR_J(\Q)\backslash\cR_J(\bA))$ and regard $\Phi(ns)$ as a function in $n\in\cN(\bA)$ for each fixed $s\in\cH_J^{0,1}(\bA)$. Furthermore let $\tilde{s}$ be an inverse image of $s$ by the covering map $\widetilde{\cH_J^{0,1}(\bA)}\rightarrow\cH_J^{0,1}(\bA)$. We take an orthogonal basis $\{h_i\}$ of $L^2_{\xi}(\cN(\Q)\backslash\cN(\bA))$, which can be taken as the image of an orthogonal basis of $L^2(\bA\oplus J_{\bA})$ by $\theta_{\xi}$. Then $\{\omega_{\xi,\bA}(\tilde{s})h_i\}$ also forms an orthogonal basis of $L^2_{\xi}(\cN(\bA)\backslash\cN(\bA))$. We remark that $\omega_{\xi,\bA}(\tilde{s})h_i$ satisfies the left $\cH_J^{0,1}(\Q)$-invariance, which is due to \cite[Th\'eor\`eme 6]{We1}. We expand $\Phi(ns)$ by this orthogonal basis to have
\[
\Phi(ns)=\sum_i\varphi_i(\tilde{s})\cdot\omega_{\xi,\bA}(\tilde{s})h_i(n),
\]
where 
\[
\varphi_i(\tilde{s}):=\int_{\cN(\Q)\backslash\cN(\bA)}\Phi(ns)\cdot\overline{\omega_{\xi,\bA}(\tilde{s})h_i(n)}dn\in L^2(\cH_J^{0,1}(\Q)\backslash\widetilde{\cH_J^{0,1}}(\bA)).
\]
The expansion of $\Phi$ above does not depend on the inverse image $\tilde{s}$ of $s$, i.e. we can replace $\tilde{s}$ by $s$. This implies that the injection mentioned above turns out to be also a surjection.
\end{proof}
By Proposition \ref{decomp-Jacobifm} $L^2_{\xi}(\cR_J(\Q)\backslash\cG_J(\bA))$ decomposes into
\[
L^2_{\xi}(\cR_J(\Q)\backslash\cR_J(\bA))=L^2_{\xi,c}(\cR_J(\Q)\backslash\cR_J(\bA))\oplus L^2_{\xi,d}(\cR_J(\Q)\backslash\cR_J(\bA)),
\]
where $L^2_{\xi,c}(\cR_J(\Q)\backslash\cG_J(\bA))\simeq L_{\xi}^2(\cN(\Q)\backslash\cN(\bA))\otimes L^2_{c}(\cH_J^{0,1}(\Q)\backslash\widetilde{\cH_J^{0,1}}(\bA))_{-}$~(respectively~\\
$L^2_{\xi,d}(\cR_J(\Q)\backslash\cR_J(\bA))\simeq L_{\xi}^2(\cN(\Q)\backslash\cN(\bA))\otimes L^2_{d}(\cH_J^{0,1}(\Q)\backslash\widetilde{\cH_J^{0,1}}(\bA))_{-}$) with the continuous spectrum $L^2_{c}(\cH_J^{0,1}(\Q)\backslash\widetilde{\cH_J^{0,1}}(\bA))_{-}$ and the discrete spectrum $L^2_{d}(\cH_J^{0,1}(\Q)\backslash\widetilde{\cH_J^{0,1}}(\bA))_{-}$. For the spectral decomposition of $L^2(\cH_J^{0,1}(\Q)\backslash\widetilde{\cH_J^{0,1}}(\bA))$ we cite well known references \cite{Mo-Wa} and \cite{La}.
\begin{defn}
We call $L^2_{\xi,c}(\cR_J(\Q)\backslash\cR_J(\bA))$~(respectively~$L^2_{\xi,d}(\cR_J(\Q)\backslash\cR_J(\bA))$) the continuous spectrum of $L^2_{\xi}(\cR_J(\Q)\backslash\cR_J(\bA))$~
(respectively~discrete spectrum of $L^2_{\xi}(\cR_J(\Q)\backslash\cR_J(\bA))$).
\end{defn}
\begin{thm}\label{adelicFJ-exp-QDS}
Suppose that $F$ be a cusp form generating a quaternionic discrete series representation at the archimedean place. Let $\xi\in\Q\setminus\{0\}$.
Suppose that $F$ generates a $K$-type $(\tau, V)$ of the quaternionic discrete series as a $K$-module. For each fixed $g\in\cG(\bA)$, as a function in $h\in\cR_J(\bA)$, we have 
\[
(F_{\xi}(hg),w)\in L^2_{\xi,c}(\cR_J(\Q)\backslash\cR_J(\bA))
\]
for $w\in V$ with the $K$-invariant inner product $(*,*)$ of $V$.
\end{thm}
\begin{proof}
To verify the assertion suppose that $(F_{\xi}(hg),w)$ has non-zero contribution from $L^2_{\xi,d}(\cR_J(\Q)\backslash\cR_J(\bA))$. Then the restriction of $F_{\xi}$ to the archimedean component gives rise to the intertwining operator from a quaternionic discrete series representation to a $C^{\infty}$-induction from an irreducible admissible representation of the real Jacobi group $R_J$ as in Theorem \ref{FJ-model}.  For this we note that $L^2_{\xi,d}(\cR_J(\Q)\backslash\cR_J(\bA))$ decomposes discretely into a sum of irreducible unitary representations of $\cR_J(\bA)$ given by outer tensor products $\omega_{\xi,\bA}\boxtimes \sigma$ with irreducible unitary representations $\sigma$ of $\widetilde{\cH_J^{0,1}}(\bA)$ since  $L^2_d(\cH_J^{0,1}(\Q)\backslash\widetilde{\cH_J^{0,1}}(\bA))$ decomposes into a discrete sum of irreducible unitary representations. In addition to this, we further remark that the representations $\omega_{\xi,\bA}\boxtimes \sigma$ above admit decompositions into pure tensor products of irreducible admissible representations over local fields, for which note that the representations at the  archimedean place should be unitary since automorphic forms in $\sigma$ are verified to be square-integrable in the non-adelic formulation. However, what we have supposed contradicts to the non-existence of the Fourier-Jacobi model for quaternionic discrete series of $G$ in Theorem \ref{FJ-model}. 
\end{proof}

\section{Application:~cusp forms constructed by Pollack}\label{Application}
We first provide a general fact on theta correspondences for automorphic representations. 
\begin{prop}\label{TheteCorresp}
Let $({\cal G},{\cal H})$ be a pair of reductive algebraic groups over a number field unramified at every non-archimedean place~(see the definition of ``unramified'' for \cite[Section 7]{H1}) and suppose that they form a dual pair, for which we assume that ${\cal H}$ is an orthogonal group of even degree when $({\cal G},{\cal H})$ is a symplectic-orthogonal dual pair. 
Let $F$ be a cusp form on the adelization of ${\cal G}$ and let $F$ generate an irreducible cuspidal representation, say $\pi_F$. We denote the theta lift of $\pi_F$ by $\theta(\pi_F)$. 

Suppose that $F$ is unramified at every non-archimedean place and that, at every archimedean place $v$, $\theta(F)$ generates an irreducible $K_v$-representation as a $K_v$-module, with a maximal compact subgroup $K_v$ of the completion of ${\cal H}$ at $v$. 
Then, if $\theta(F)$ is cuspidal and unramified at every non-archimedean place, $\theta(F)$ generates an irreducible cuspidal representation for ${\cal H}$. 
\end{prop}
\begin{proof}
We first recall that the cuspidal spectrum decomposes discretely with finite multiplicities. This implies that the cuspidal representation $\theta(\pi_F)$ generated by $\theta(F)$ is reducible and a finite sum of irreducible cuspidal representations of ${\cal H}$, from which the same property~(reducibility and the finiteness of the sum) holds for every local component of $\theta(\pi_F)$. 

Now review a general theory of local theta correspondences by Howe \cite{H1} and \cite{H2}.  
The well known conjecture by Howe says that a local theta lift from an irreducible admissible representation has a unique irreducible subquotient, which is proved by  \cite[Theorem 7.1]{H1}~(respectively~\cite[Theorem A1]{H2}) for unramified representations over non-archimedean fields~(respectively~the archimedean case). In view of this we see that every local component of $\theta(\pi_F)$ is a finite copy of one irreducible admissible representation, which is unramified at each non-archimedean place by \cite[Theorem 7.1]{H1}. 
From this we deduce that $\theta(F)$ should be a Hecke-eigen vector unramified at every non-archimedean place. Here note that an unramified vector of an irreducible uniramified representation is a Hecke-eigen vector as is well known. 
On the other hand, by the assumption at archimedean places for $\theta(F)$, $\theta(F)$ should be inside one irreducible representation. As a result of \cite[Theorem 3.1]{NPS} $\theta(F)$ generates an irreducible cuspidal representation. 
\end{proof}
We apply the similar argument of the proof above to Pollack's theta lift~(cf.~\cite[Theorem 4.1.1]{P3})  to the adele group $SO(4,4)(\bA)$ of $SO(4,4)$ defined by the quadratic space over $\Q$ unimodular at every finite places, for which we remark that Pollack deals with theta lifts to $SO(4,8k+4)$ for $k\ge 0$. However, let us note that the representation type of $\theta(F)$ at the archemedean place is not determined in \cite{P3}. For the case of $k=0$, which is related to the exceptional group $G_2$, we have the following:
\begin{prop}
For a holomorphic Siegel cusp form $F$ of degree two with respect to the Siegel modular group $Sp(4,\Z)$, let $\theta(F)$ be the theta lift from $F$ to the adele group $SO(4,4)(\bA)$ given by Pollack above. 
\begin{enumerate}
\item Suppose that $F$ is a Hecke eigenform. Then $\theta(F)$ is a Hecke-eigen cusp form and generates a quaternionic discrete series representation if $\theta(F)\not\equiv 0$. 
In fact, $\theta(F)$ generates an irreducible cuspidal representation of the adele group $SO(4,4)(\bA)$.
\item Suppose that $F$ has the non-zero Fourier coefficient indexed by 
$\begin{pmatrix}
1 & 0\\
0 & 1
\end{pmatrix}$. 
Let us consider the restriction of $\theta(F)$ to the adele group of the exceptional group $G_2$ as in Pollack \cite{P3}. This is a non-zero cusp form on the adele group $G_2(\bA)$ of $G_2$ generating a quaternionic discrete series representation at the archimedean place, which obviously implies the non-vanishing of $\theta(F)$ itself. 
\end{enumerate}
\end{prop}
\begin{proof}
In Pollack \cite[Corollary 3.3.7]{P3} and \cite[Corollary 4.2.3]{P3} it is already shown that $\theta(F)$ and $\theta(F)|_{G_2(\bA)}$ are cuspidal quaternionic form on $SO(4,4)(\bA)$ and $G_2(\bA)$ respectively. Proposition \ref{module_QMF} then tells us that these cusp forms generate quaternionic discrete series representations. 

Furthermore, from the unramified data to construct $\theta(F)$ as in \cite[Theorem 4.1.1]{P3}, we see that $\theta(F)$ is unramified at every finite place. Let us note that $\theta(F)$ can be veiwed as the restriction of an automorphic form on $O(4,4)(\bA)$ unramified at every finite place.
We note that a Hecke-eigen holomorphic Siegel cusp form of full level generates an irreducible cuspidal represnetation~(cf.~\cite[Corolloary 3.2]{NPS}). As we have seen in Proposition \ref{TheteCorresp}, the cuspidal representation generated by the global theta lift $\theta(F)$ to the adele group $O(4,4)(\bA)$ is irreducible, and we see that every local component is therefore irreducible. The restriction of the local component to $SO(4,4)(\Q_v)$ for each place $v$ is a sum of an irreducible admissible representation and its conjugate by the non-trivial element of the $O(4,4)(\Q_v)/SO(4,4)(\Q_v)$, both of which are irreducible unramified representations with the same Hecke eigenvalue. From the unramified-ness of $\theta(F)$ mentioned above we see that $\theta(F)$ is a Hecke eigenform at every non-archimedean place since an unramified vector of an irreducible unramified representation is a Hecke-eigen vector. 

Therefore, reviewing \cite[Theorem 3.1]{NPS}, we see that $\theta(F)$ generates an irreducible cuspidal representation of $SO(4,4)(\bA)$.

We are left with the assertion on the non-vanishing of $\theta(F)|_{G_2(\bA)}$ under the assumtion of the 2nd assertion. This is an immediate consequence of \cite[Corollary 4.2.4]{P3}). 
We are therefore done. 
\end{proof}

\appendix
\section{The non-adelic Fourier-Jacobi expansion of cusp forms generating quaternionic discrete series}\label{FJ-exp-QDS}
To discuss the Fourier-Jacobi expansion in the non-adelic setting 
we suppose that we are given an arithmetic subgroup $\Gamma$ of $G$, contained in $\cG(\Q)$.
%

Before the statement of the theorem below we note that the set of unitary characters of $N$, which factor through $N/N_0\simeq\R\oplus J\oplus J^{\vee}\oplus \R$, can be viewed as the quadratic space equipped with the symplectic form $\langle*,*\rangle$~(cf.~Section \ref{Cubicforms-Groups}). Recall that each element of $\R\oplus J\oplus J^{\vee}\oplus \R$ has been written as $(a,b,c,d)$, where $a,d\in\R,~b\in J,~c\in J^{\vee}$. Characters of $N/N_0$ parametrized by $w'\in\R\oplus J\oplus J^{\vee}\oplus \R$ are given by
\[
N/N_0\simeq \R\oplus J\oplus J^{\vee}\oplus \R\ni w\mapsto\exp(2\pi\sqrt{-1}\langle w',w\rangle).
\]
To state the theorem we introduce $N_{\Gamma}:=N\cap\Gamma$ and let $L_{\Gamma}^*$ be the dual lattice of $L_{\Gamma}:=\{a\in\Q\mid (0;a)\in N_{\Gamma}\}$, given by $L_{\Gamma}^*=\{a'\in\Q\mid a'a\in\Z~\forall a\in L_{\Gamma}\}$. In addition to this, we make the following assumption on $\Gamma$.
\begin{asm}
The reflection $w_{\alpha}$ with respect to $\alpha$ belongs to $\Gamma$, and $H_J^{0,1}\cap\Gamma$ stabilizes $N_{\Gamma}$.
\end{asm}
\noindent
For example, the level one subgroup of $G_2$ in Gan-Gross-Savin \cite{G-G-S} satisfies this assumption.

The theorem we are going to state is about the Fourier-Jacobi expansion of a cusp form $F$ given by
\[
F(g)=\sum_{\xi\in L_{\Gamma}^*}F_{\xi}(g),
\]
where $F_{\xi}$ denotes the Fourier transformation of $F$ with respect to the central character of $N$ indexed by $\xi\in L_{\Gamma}^*$. 

Let us think of $F_{\xi}$ in detail for $\xi\in L_{\Gamma}^*\setminus\{0\}$. 
From \cite[Theorem 5.1,~Sections 5,~6]{Co-G-1} we now recall that $\Hom_N(\eta_{\xi},L^2(N_{\Gamma}\backslash N))$ is finite dimensional and thus 
\[
L^2_{\xi}(N_{\Gamma}\backslash N)\simeq\Hom_N(\eta_{\xi},L^2_{\xi}(N_{\Gamma}\backslash N))\otimes V_{\xi}\simeq\oplus_{i=1}^{m(\xi)}\theta_i(V_{\xi}),
\]
where $\{\theta_i\mid 1\le i\le m(\xi)\}$ denotes a basis of $\Hom_N(\eta_{\xi},L^2_{\xi}(N_{\Gamma}\backslash N))$. We have a more detailed description of $\{\theta_i\mid 1\le i\le m(\xi)\}$. By applying \cite[Theorem 5.1,~Sections 5,~6]{Co-G-1}~(see also \cite[Section 6.2, Propositions 6.6,~6.7]{Na-1}) to our situation we have the following: 
\begin{prop}\label{AutomFm-HeisenbergGp}
Let $\xi\in L_{\Gamma}^*\setminus\{0\}$. We put $\Lambda_{\Gamma}:=\{\lambda=(a,b,c,d)\in W_J\mid (\lambda;t_{\lambda})\in N_{\Gamma}~\exists t_{\lambda}\in\Q\}$ and $\hat{\Lambda}_{\Gamma,\xi}:=\{\nu\in W_J(\Q)\mid \xi\langle\nu,\lambda\rangle\in\Z~\forall\lambda\in \Lambda_{\Gamma}\}$, the latter of which is the dual lattice of $\Lambda_{\Gamma}$ defined by the bilinear form $\xi\langle*,*\rangle$. With the projection ${\rm Pr}:W_J\ni (a,b,c,d)\mapsto (a,b,0,0)\in W_J^0:=\{(a,b,0,0)\mid (a,b)\in\R\oplus J\}$ we introduce
\[
\Lambda^0_{\Gamma}:={\rm Pr}(\Lambda_{\Gamma}),
\quad\hat{\Lambda}_{\Gamma,\xi}^0:={\rm Pr}(\hat{\Lambda}_{\Gamma,\xi}).
\]
\begin{enumerate}
\item We have $m(\xi)=\sharp(\hat{\Lambda}_{\Gamma,\xi}^0/\Lambda_{\Gamma}^0)$.
\item For each representative $\nu$ of $\hat{\Lambda}_{\Gamma,\xi}^0/\Lambda_{\Gamma}^0$ we define $\Theta_{\xi}^{\nu}\in\Hom_N(\eta_{\xi},L^2_{\xi}(N_{\Gamma}\backslash N))$ by
\[
\Theta_{\xi}^{\nu}(h)(n(w,t)):=\sum_{\gamma\in N_m\cap\Gamma\backslash N_{\Gamma}}h(n(\nu,0)\gamma n(w,t))\quad(h\in V_{\xi}\simeq L^2(\R\oplus J)),
\]
where recall that the representation space $V_{\xi}$ of $\eta_{\xi}$ can be regarded as $L^2(\R\oplus J)$.

The $\{\Theta_{\xi}^{\nu}\mid\nu\in \hat{\Lambda}_{\Gamma,\xi}^0/\Lambda_{\Gamma}^0\}$ forms a basis of $\Hom_N(\eta_{\xi},L^2_{\xi}(N_{\Gamma}\backslash N))$.
\end{enumerate}
\end{prop}
With an orthogonal basis $\{h_j^{\xi}\}$ of $V_{\xi}$ we express $F_{\xi}$ as 
\[
F_{\xi}(ug)=\sum_{i=1}^{m(\xi)}\sum_{j}\theta_i(h_j^{\xi})(u)c'_{ij}(g)\quad((u,g)\in N\times G),
\]
with $\bV_n$-valued functions $\{c'_{ij}\}$ on $G$. 
For an arbitrary fixed $s\in H_J^{0,1}$ we put 
\[
c_{ij}(sg):=(F_{\xi}(*sg),\theta(\Omega_{\xi}(s)h_j^{\xi})(*))_{\xi},
\] viewed as a function on $G$ 
, where $(*,*)_{\xi}$ denotes the canonical $\bV_n$-valued paring for $(\theta(V_{\xi})\otimes\bV_n)\times \theta(V_{\xi})\simeq(V_{\xi}\otimes\bV_n)\times V_{\xi}$ induced by the inner product of $\theta(V_{\xi})\simeq V_{\xi}$~(the inner product respects the variable $*$ for $N$). We then have another expansion
\begin{align*}
F_{\xi}(usg)=\sum_{i=1}^{m(\xi)}\sum_{j}\theta_i(\Omega_{\xi}(s)h_j^{\xi})(u)c_{ij}(sg)\quad((u,s,g)\in N\times H_J^{0,1} \times G),
\end{align*}
where we remark that $\{\Omega_{\xi}(s)h_j^{\xi}\}$ also forms an orthogonal basis of $V_{\xi}$~(for the notation $\Omega_{\xi}$ see Section \ref{Rep-JacobiGp}). 

We now denote by $\widetilde{H_J^{0,1}}$ the non-split two fold cover if $\Omega_{\xi}|_{H_j^{0,1}}$ has non-trivial multiplier. Otherwise we put $\widetilde{H_J^{0,1}}:=H_J^{0,1}$. For a subgroup $S$ of $H_J^{0,1}$, let $\tilde{S}$ be the pull-back of $S$ by the covering map $\widetilde{H_J^{0,1}}\rightarrow H_J^{0,1}$.
\begin{lem}\label{Sq-int_corff-fct}
With any fixed $g\in G$ the function $c_{ij}(sg)$s above in $s\in \widetilde{H_J^{0,1}}$ belong to $L^2(\widetilde{\Gamma_0}\backslash\widetilde{H_J^{0,1}})\otimes {\mathbb V}_{n}$ as functions in $s\in\widetilde{H_J^{0,1}}$ with a subgroup $\Gamma_0$ of finite index in $H_J^{0,1}\cap\Gamma$.
\end{lem}
\begin{proof}
Following the argument by Shintani \cite[Proposition 1.1]{Sh}, the Weil representation $\omega_{\xi}$ of $Mp(W_J)$ induces a finite dimensional  representation of $\widetilde{H_J^{0,1}\cap\Gamma}$ on $\Hom_N(\eta_{\xi},L^2_{\xi}(N_{\Gamma}\backslash N)$ by
\[
(r_{\xi}(\gamma)\theta)(h):=\theta(\Omega_{\xi}(\gamma)\cdot h)\quad(\gamma\in \widetilde{H_J^{0,1}\cap\Gamma})
\]
for $(\theta,h)\in \Hom_N(\eta_{\xi},L^2_{\xi}(N_{\Gamma}\backslash N))\times V_{\xi}$, where we note that the assumption on $H_J^{0,1}\cap\Gamma$ is necessary for the representation $r_{\xi}$ to be well defined. 
This representation has a compact image siting inside a compact unitary group. 
We note that $\{\theta_i\}$ are regarded as theta series and can be chosen so that they correspond to the theta series as in $\cite[Proposition 1.1]{Sh}$. 
The discrete group $\Gamma_0$ is taken so that $\widetilde{\Gamma_0}$ is the kernel of the representation $r_{\xi}$ of $\widetilde{H_J^{0,1}\cap\Gamma}$ on  
$\Hom_N(\eta_{\xi},L^2_{\xi}(N_{\Gamma}\backslash N))$. 
We see $[H_J^{0,1}\cap\Gamma,\Gamma_0]<\infty$ since the representation $r_{\xi}$ of $\widetilde{H_J^{0,1}\cap\Gamma}$ has a compact image as is mentioned above. 

From the left $H_J^{0,1}\cap\Gamma$ invariance of $F_{\xi}$ we verify
\begin{align*}
F_{\xi}(usg)=F_{\xi}((\gamma u\gamma^{-1})(\gamma s)g)&=\sum_{i=1}^{m(\xi)}\sum_{j}c_{ij}((\gamma s)g)\theta_i(\Omega_{\xi}(\gamma s)h_j)(\gamma u\gamma^{-1})\\
&=\sum_{i=1}^{m(\xi)}\sum_{j}c_{ij}((\gamma s)g)(r_{\xi}(\gamma)\cdot \theta_i)(\Omega_{\xi}(s)h_j)(u)~(\forall\gamma\in H_J^{0,1}\cap\Gamma)
\end{align*}
by the same reasoning as in \cite[Lemma 5.9]{Na-2}. 
Here, as in the proof \cite[Lemma 5.9]{Na-2}, we use the fact that $\theta_i$ exchanges $\eta_{\xi}$ with the right regular representation of $N$ on $L^2_{\xi}(N_{\Gamma}\backslash N)$, together with the fundamental property of the Weil representation as follows:
\[
\eta_{\xi}(u)\Omega_{\xi}(s)=\Omega_{\xi}(s)\eta_{\xi}(s^{-1}us).
\]
We therefore see that
\[
\sum_{i=1}^{m(\xi)}c_{ij}(\gamma s)(r_{\xi}(\gamma)\cdot \theta_i)=\sum_{i=1}^{m(\xi)}c_{ij}(s)\theta_i.
\]
For $\gamma\in \Gamma_0$ this implies the left $\Gamma_0$-invariance of $c_{ij}$. 
This invariance can be extended to $\widetilde{\Gamma_0}$. 
\end{proof}
As is well known, the $L^2$-space $L^2(\widetilde{\Gamma_0}\backslash\widetilde{H_J^{0,1}})$ decomposes into the orthogonal sum of the discrete spectrum $L^2_{ds}(\widetilde{\Gamma_0}\backslash\widetilde{H_J^{0,1}})$ and the continuous spectrum $L^2_{ct}(\widetilde{\Gamma_0}\backslash\widetilde{H_J^{0,1}})$. The argument of the proof for the lemma above implies that, for $\xi\in L_{\Gamma}^*\setminus\{0\}$, $F_{\xi}$ admits a decomposition $F_{\xi}(g)=F_{\xi,ct}(g)+F_{\xi,ds}(g)~(g\in G)$, where, for any fixed $g\in G$,  
\[
F_{\xi,ct}(rg)\in L^2_{ct}(\widetilde{\Gamma_0}\backslash\widetilde{H_J^{0,1}})\otimes\oplus_{i=1}^{m(\xi)}\theta_i(V_{\xi}),\quad F_{\xi,ds}(rg),\in L^2_{ds}(\widetilde{\Gamma_0}\backslash\widetilde{H_J^{0,1}})\otimes\oplus_{i=1}^{m(\xi)}\theta_i(V_{\xi})
\]
as a function in $r\in R_J$. We then arrive at the following:
\begin{defn}
Let $F_{\xi}$, $F_{\xi,ct}$ and $F_{\xi,ds}$ be as above for $\xi\in L_{\Gamma}^*\setminus\{0\}$. We call $F_{\xi,ds}$ and $F_{\xi,ct}$ the discrete part and the continuous part respectively.
\end{defn}
To state the theorem below we further introduce 
\[
\hat{\Lambda}_{\Gamma,\xi}^{(a)}:={\rm Pr}^{(a)}(\hat{\Lambda}_{\Gamma,\xi})
\]
with ${\rm Pr}^{(a)}:W_J\rightarrow W_J^{(a)}:=\{(a,0,0,0)\in W_J\mid a\in\R\}$.
\begin{thm}\label{main-theorem}
Let $F(g)=\sum_{\xi\in L_{\Gamma}^*}F_{\xi}(g)$ be the Fourier-Jacobi expansion of a cusp form $F$ on $G$ with respect to $\Gamma$~(satisfying the assumption above) generating quaternionic discrete series representation. 
\begin{enumerate}
\item For $\xi\in L_{\Gamma}^*\setminus\{0\}$ we have $F_{\xi,ds}\equiv 0$, namely 
\[
F_{\xi}(g)\equiv F_{\xi,ct}(g)
\]
for $g\in G$. 
\item For a character $\chi$ of $N\cap\Gamma\backslash N$ let $F_{\chi}$ denote  the Fourier transformation of $F$ by $\chi$. Then $F_{\xi}(g)\equiv F_{\xi,ct}(g)$ admits an expansion
\[
F_{\xi}(g)=\sum_{\nu\in\hat{\Lambda}^{(a)}_{\Gamma,\xi}}\sum_{\chi_{\xi}}F_{\chi_{\xi}}(w_{\alpha}n(\nu,0)g)=\sum_{\nu'\in\hat{\Lambda}^0_{\Gamma,\xi}/\Lambda_{\Gamma}^0}\sum_{\chi'_{\xi}}\Theta_{\chi'_{\xi}}^{\nu'}(g),
\]
with
\[
\Theta_{\chi'_{\xi}}^{\nu'}(g):=\sum_{\gamma\in N_m\cap\Gamma\backslash N_{\Gamma}}F_{\chi'_{\xi}}(w_{\alpha}n(\nu',0)\gamma g) 
\]
where $\chi_{\xi}$~(respectively~$\chi'_{\xi}$) runs over characters of $N_{\Gamma}\backslash N$ parametrized by $(\xi,b',c',d')\in\hat{\Lambda}_{\Gamma,\xi}$~(respectively $(\xi,0,c',d')\in\hat{\Lambda}_{\Gamma,\xi}$) of rank four satisfying the negativity with respect to the Freudenthal's quartic form~(see Pollack \cite[Theorem 1.2.3]{P1}).
\end{enumerate}
\end{thm}
\begin{proof}
1.~To see the first assertion we note that $L^2_{ds}(\widetilde{\Gamma_0}\backslash \widetilde{H_J^{0,1}})$ decomposes into a discrete sum of irreducible unitary representations of $\widetilde{H_J^{0,1}}$ with finite multiplicities. In view of Lemma \ref{Sq-int_corff-fct}, $\sum_{j}\theta_{i}(\Omega_{\xi}(s)h_j^{\xi})c_{ij}(sg)$ can be viewed as $V_{\xi}\boxtimes L^2(\widetilde{\Gamma_0}\backslash\widetilde{H_J^{0,1}})\boxtimes \bV_n$-valued function in $g\in G$ satisfying the left $R_J$-equivariance by $\rho_{\sigma_0,\xi}$ and right $K$-equivariance by $\tau_n$, where $\rho_{\sigma_0,\xi}:=\sigma_0\otimes\Omega_{\xi}$ with the right regular representation $\sigma_0$ of $\widetilde{H_J^{0,1}}$ on $L^2(\widetilde{\Gamma_0}\backslash\widetilde{H_J^{0,1}})$. So it is determined by its restriction to the connected torus subgroup $A_P\simeq\{w\in\R^{\times}\mid w>0\}$ in the center of $H_J$ .
With a basis $\{v_{k}^{2n}\mid -n\le k\le n\}$~(e.g. that given in Section \ref{GWF-STrep}) of the representation space $\bV_n$ of $\tau_n$, let $c_{ij,k}$ be the coefficient function of $c_{ij}$ for $v_{k}^{2n}$ and suppose that $c_{ij,k}\in L^2_{ds}(\widetilde{\Gamma_0}\backslash\widetilde{H_J^{0,1}})$. Given any irreducible unitary representation $(\pi,V_{\pi})$ with the representation space $V_{\pi}$ in $L^2_{ds}(\widetilde{\Gamma_0}\backslash\widetilde{H_J^{0,1}})$, let $\{u_l\}\subset L^2_{ds}(\widetilde{\Gamma_0}\backslash\widetilde{H_J^{0,1}})$ be an orthogonal basis of $V_{\pi}$. We can let $a_{ij,k}^{(l)}(g):=\langle c_{ij,k}(*g),u_l\rangle$ with the Petersson inner product $\langle *,*\rangle$ of $L^2(\widetilde{\Gamma_0}\backslash\widetilde{H_J^{0,1}})$. Now let $h_g=n_gs_g$ be the $R_J$-factor of $g\in G$ with $n_g\in N$ and $s_g\in H_J^{0,1}$. We then see that, with this $a_{ij,k}^{(l)}$, we can write 
\[
\sum_{k}a_{ij,k}^{(l)}(g)v_k^{2n}=\int_{\Gamma_J^0\backslash R_J}\overline{u_l(s)\theta_i(\Omega_{\xi}(s)h_j^{\xi})(nsn_gs^{-1})}F(nsg)dh~(h=ns, (n,s)\in N\times H_J^{0,1}),
\]
where $\Gamma_J^0:=N_{\Gamma}\rtimes\Gamma_0$ and $dh$ denotes a right invariant measure of $\Gamma_J^0\backslash R_J$. For this note that $nsn_g=n(sn_gs^{-1})s$. 
The $\bV_n$-valued functions on $G$ above are annihilated by the Dirac-Schmid operator since so is $F$. Therefore, in view of Proposition \ref{module_QMF}, they generate a quaternionic discrete series as $F$ does.
 
Furthermore let us note that $\theta_i(f)(xn)$ with $f\in V_{\xi}$ means the right translate of $\theta_i(f)(x)$ by $n\in N$, where $x$ denotes a variable for $N$. 
If there is an irreducible unitary representation $\pi$ occuring in $L^2_{ds}(\widetilde{\Gamma_0}\backslash\widetilde{H_J^{0,1}})$ with $u_l\in V_{\pi}$ such that $a_{ij,k}^{(l)}\not\equiv 0$, we have a non-zero element in the image of $\Hom_{(\g,K)}(\pi_n,C^{\infty}_{\eta_{\xi}}(N\backslash G))\rightarrow\Hom_K(\pi_n,C^{\infty}_{\rho_{\pi,\xi}}(N\backslash G))$~(cf.~Section \ref{GWF-definition}) by 
\[
F\mapsto \theta_i(h_j^{\xi})(xn_g)(\sum_{k}a_{ij,k}^{(l)}(g)v_k^{2n}),
\]
where we regard $\theta_i(\Omega_{\xi}(s_g)h_j^{\xi})(xn_g)$~(respectively~$F$) as an element in $V_{\xi}\simeq \theta_i(V_{\xi})$~(respectively~$\pi_n$ viewed as a subrepresentation of cusp forms on $G$). 
We can verify that the image of $F$ is viewed as a well-defined $V_{\xi}$-valued function as we did for 
\[
(\Phi(\pi_n(g)v),\rho_{\sigma,\xi}(h_g)u_0)_{\rho_{\sigma,\xi}}(\rho_{\sigma,\xi}(n_g)u_0,u_{00})_{\sigma}
\]
in the proof of Theorem \ref{FJ-model}. 
This satisfies the separation of variables in the sense of Section \ref{Pf-Prop2.4} and the image of the associated intertwining operator is of moderate growth with respect to $w$. However, this contradicts to Theorem \ref{GenWhittakerFct-gen}. Therefore $c_{ij}\equiv 0$ for any $i,j$. Taking this fact into account,  Theorem \ref{FJ-model} implies that the contribution of $L^2_{ds}(\widetilde{\Gamma_0}\backslash\widetilde{H_J^{0,1}})$ to the Fourier-Jacobi expansion is zero, which leads to the assertion.\\[5pt]
2.~
We are left with considering the contribution by $L^2_{ct}(\widetilde{\Gamma_0}\backslash\widetilde{H_J^{0,1}})$ to the Fourier-Jacobi expansion. 
We remark that the argument of the proof of the first assertion has used the generalized Whittaker functions attached to the Schr\"odinger representation $\eta_{\xi}$ in order to deny the contribution of $L^2_{ds}(\widetilde{\Gamma_0}\backslash\widetilde{H_J^{0,1}})$ to the Fourier-Jacobi expansion. However, irreducible unitary representations cannot occur in $L^2_{ct}(\widetilde{\Gamma_0}\backslash\widetilde{H_J^{0,1}})$ as its subrepresentations. The idea verifying the first assertion cannot therefore explain how $L^2_{ct}(\widetilde{\Gamma_0}\backslash\widetilde{H_J^{0,1}})$ contribute to the Fourier-Jacobi expansion with such generalized Whitaker functions. 

For the proof of the second assertion, instead, we have an idea to express $F_{\xi}~(=F_{\xi,ct})$ in terms of the generalized Whittaker functions attached to characters of $N$. We begin with expanding $F_{\xi}$ along $N_m\cap\Gamma\backslash N_m$ to get
\[
F_{\xi}(g)=\sum_{\psi}F_{\psi}(g),\quad F_{\psi}(g):=\int_{N_m\cap\Gamma\backslash N_m}F(xg)\psi(x)^{-1}dx,
\]
where $\psi$ runs over characters of $N_m\cap\Gamma\backslash N_m$ such that $\psi|_{N_0}=\psi_{\xi}|_{N_0}$ with the center $N_0$ of $N$. 

To proceed further we need the following lemma, which is the non-adelic version of Lemma \ref{adelic_Ch-Lem}:
\begin{lem}\label{non-adelic_Ch-Lem}
\begin{enumerate}
\item Any character $\psi$ above is of the form
\[
N_m\ni (w,t)\mapsto\psi_{\xi}((\nu+\mu;0)(w;t)(-(\nu+\mu);0))=\exp(2\pi\sqrt{-1}\xi(\langle\nu+\mu,w\rangle+t))
\]
with $\nu\in\hat{\Lambda}^0_{\Gamma,\xi}/\Lambda^0_{\Gamma}$ and $\mu\in\Lambda^0_{\Gamma}$. 
\item We have $w_{\alpha}N_mw_{\alpha}^{-1}=N_m$.
\end{enumerate}
\end{lem}.
\noindent
In view of this lemma and the bijection $\Lambda_{\Gamma}^0\simeq N_m\cap\Gamma\backslash N_{\Gamma}$ we have
\[
F_{\xi}(g)=\sum_{\nu\in\hat{\Lambda}^0_{\Gamma,\xi}/\Lambda_{\Gamma}^0}\sum_{\gamma\in N_m\cap\Gamma\backslash N_{\Gamma}}F_{\psi_{\xi}}(n(\nu,0)\gamma g),
\]
also have
\[
F_{\xi}(g)=\sum_{\nu\in\hat{\Lambda}_{\Gamma,\xi}^{(a)}}\sum_{\psi_{\xi}^{(b)}}F_{\psi_{\xi}^{(b)}}(n(\nu,0)g).
\]
Here $\psi_{\xi}^{(b)}$ runs over characters of $N_m$ given by $\exp(2\pi\sqrt{-1}\xi(\langle\nu,w\rangle+t))$ with $\nu\in {\rm Pr}^{(b)}(\hat{\Lambda}_{\Gamma,\xi})$, where 
${\rm Pr}^{(b)}:W_J\ni (a',b',c',d')\mapsto (0,b',0,0)\in W_J^{(b)}:=\{(0,b,0,0)\in W_J\mid b\in J\}$.

As the next step 
we should now note that the adjoint action by $w_{\alpha}$ on the Lie algebra of $G$~(induced by $w_{12}$) gives rise to exchanging of the center $N_0$ of $N$ and the last entry of the vector part $W_J$ of $N$, which correspond  to $E_{13}$ and $E_{23}$ respectively in the ${\mathfrak sl}_3$-factor of the $\Z/3$-grading of the Lie algebra. For a further detail on this see the proof of the second assertion of Lemma \ref{adelic_Ch-Lem}. We can thus consider the expansion of $F_{\psi_{\xi}}(w_{\alpha}*)$ by characters of $N$ since this is invariant with respect to the left translation by the center $N_0$. 
In order that the $w_{\alpha}$-conjugate of a character parametrized by $w'=(a',b',c',d')$ has a contribution to $F_{\xi}(w_{\alpha}*)$, we need $a'\not=0$, which corresponds to the non-triviality of the $w_{\alpha}$-twisted central character of $N$ for $F_{\xi}$ via the symplectic form $\langle*,*\rangle$ on $\R\oplus J\oplus J^{\vee}\oplus\R\simeq N/N_0$. 

We are now going to verify that
\[
F_{\psi_{\xi}}(w_{\alpha}g)=\sum_{\chi'_{\xi}}F_{\chi'_{\xi}}(g),\quad\sum_{\psi_{\xi}^{(b)}}F_{\psi_{\xi}^{(b)}}(w_{\alpha}g)=\sum_{\chi_{\xi}}F_{\chi_{\xi}}(g)
\]
where, in view of the cuspidality of $F$, $\chi'_{\xi}$~(respectively~$\chi_{\xi}$) runs over characters of $N_{\Gamma}\backslash N$ parametrized by
$(\xi,0,c',d')$~(respectively~$(\xi,b',c',d'))$)$\in\hat{\Lambda}_{\Gamma,\xi}$ of rank four and satisfying the negativity with respect to the Freudenthal's quartic form as Pollack \cite[Theorem 1.2.1]{P1} shows. 
To verify this let us note $w_{\alpha}N_mw_{\alpha}^{-1}=N_m$ by Lemma \ref{non-adelic_Ch-Lem} part 2. We then see
\[
F_{\psi_{\xi}}(w_{\alpha}g)=F_{\psi_{\xi,\alpha}}(g),~F_{\psi_{\xi}^{(b)}}(w_{\alpha}g)=F_{\psi_{\xi,\alpha}^{(b)}}(g)
\]
with $\psi_{\xi,\alpha}(u)=\psi_{\xi}(w_{\alpha}uw_{\alpha}^{-1})$ and $\psi_{\xi,\alpha}^{(b)}(u)=\psi_{\xi}^{(b)}(w_{\alpha}uw_{\alpha}^{-1})$ for $u\in N_m$, and further note that $\psi_{\xi,\alpha}$ and $\psi_{\xi,\alpha}^{(b)}$ are trivial on the center $N_0$. 
We then show that 
\[
F_{\psi_{\xi}}(w_{\alpha}g)=\sum_{\chi'_{\xi}}F_{\chi'_{\xi}}(g),\quad\sum_{\psi_{\xi}^{(b)}}F_{\psi_{\xi}^{(b)}}(w_{\alpha}g)=\sum_{\chi_{\xi}}F_{\chi_{\xi}}(g)
\]
with $\chi_{xi}$ and $\chi'_{\xi}$ as in the assertion by considering the expansion of $F_{\psi_{\xi}}$ and $F_{\psi_{\xi}^{(b)}}$ along $\R\oplus J\simeq N_m\backslash N$. 
We have consequently verified the second assertion. 
\end{proof}

Hiro-aki Narita\\
Department of Mathematics\\
Faculty of Science and Engineering\\
Waseda University\\
3-4-1 Okubo, Shinjuku-ku, Tokyo 169-8555\\
JAPAN\\
E-mail:~hnarita@waseda.jp
\end{document}